\documentclass{amsart}
\usepackage[T1]{fontenc} 
\usepackage{amsmath,amsfonts,amsthm,mathtools,amssymb} 
\usepackage[left=1in,right=1in,top=1in,bottom=1in]{geometry}
\usepackage{dirtytalk}
\usepackage{stmaryrd}
\usepackage{amsaddr}

\usepackage[sc]{mathpazo}          
\usepackage[euler-digits]{eulervm} 

\usepackage[dvipsnames]{xcolor}
\usepackage[pagebackref]{hyperref} 
\hypersetup{
	colorlinks=true,
	linkcolor=blue,
	citecolor=BrickRed,
	urlcolor=MidnightBlue,
	pdftitle={Multivariate de Rham Comparison Theorem},
}
\usepackage{graphicx} 
\usepackage{tikz-cd}
\usetikzlibrary{decorations.pathreplacing,angles,quotes}

\newtheorem{proposition}{Proposition}[section]
\newtheorem{lemma}[proposition]{Lemma}
\newtheorem{remark}[proposition]{Remark}
\newtheorem{corollary}[proposition]{Corollary}
\newtheorem{definition}[proposition]{Definition}
\newtheorem{theorem}[proposition]{Theorem}

\newtheorem{conjecture}[proposition]{Conjecture}

\DeclareMathOperator{\Frac}{Frac}
\DeclareMathOperator{\spm}{Spm}
\DeclareMathOperator{\spec}{Spec}
\DeclareMathOperator{\spmin}{Spmin}
\newcommand{\mAinf}{A_{\mathrm{inf},\Delta}}
\newcommand{\bn}{\textbf{n}}
\newcommand{\bm}{\textbf{m}}
\newcommand{\bp}{\textbf{p}}
\newcommand{\ba}{\textbf{a}}
\newcommand{\bi}{\textbf{i}}

\newcommand{\mCHT}{C_{HT,\Delta}}

\DeclareMathOperator{\Ainf}{A_{inf}}

\DeclareMathOperator{\id}{id}

\DeclareMathOperator{\C}{\textbf{C}}
\DeclareMathOperator{\Ann}{Ann}
\DeclareMathOperator{\gr}{gr}
\DeclareMathOperator{\rank}{rank}

\DeclareMathOperator{\cH}{H}
\DeclareMathOperator{\im}{im}
\DeclareMathOperator{\Bdr}{B_{dR}}
\newcommand{\mBdrp}{B^+_{\mathrm{dR},\Delta}}
\newcommand{\pmBdrp}{B^{+,(p)}_{\mathrm{dR},\Delta}}
\DeclareMathOperator{\Bdrp}{B^+_{dR}}
\newcommand{\mBdr}{B_{\mathrm{dR},\Delta}}
\newcommand{\pmBdr}{B^{(p)}_{\mathrm{dR},\Delta}}
\newcommand{\mBHT}{B_{HT,\Delta}}
\newcommand{\pmBHT}{B^{(p)}_{HT,\Delta}}
\newcommand{\mbDHT}{\textbf{D}_{HT,\Delta}}
\newcommand{\mbDdR}{\textbf{D}_{dR,\Delta}}
\DeclareMathOperator{\BHT}{B_{HT}}

\DeclareMathOperator{\G}{G}

\DeclareMathOperator{\Fil}{Fil}

\DeclareMathOperator{\Rep}{Rep}
\DeclareMathOperator{\GL}{GL}
\DeclareMathOperator{\diag}{diag}
\DeclareMathOperator{\bD}{\textbf{D}}

\begin{document}

\title{Multivariate Period Rings}

\author{Rohit Pokhrel$^{1,2}$}
\address{$^1$Harish-Chandra Research Institute, Chhatnag Road, Jhunsi, Prayagraj, Uttar Pradesh, India}
\address{$^2$Homi Bhabha National Institute, Training School Complex, Anushakti Nagar, Mumbai 400094, India}
\email{rohitpokhrel@hri.res.in}

\begin{abstract}
    In this article, we present a new approach to studying multivariate period rings that is more consistent with classical theory and provides a clearer description of their structure. We also prove that the category of $B$-admissible representations forms a Tannakian subcategory of the category of representations of $\G_{K,\Delta}$ by defining an analogue of $(F,G)$-regular rings, which is central to the classification of representations in multivariate $p$-adic Hodge theory.
\end{abstract}

\maketitle

\section{Introduction}

The classical $p$-adic Hodge theory of Fontaine uses various period rings to classify continuous Galois representations of a $p$-adic field $K$. A period ring $B$ is a commutative ring equipped with a continuous action of $G_K$ such that $g(b_1+b_2)=g(b_1)+g(b_2)$ and $g(b_1b_2)=g(b_1)g(b_2)$ for all $g\in G_K$ and $b_1,b_2\in B$. These rings are constructed to satisfy specific properties. Let $F$ be a closed subfield of $B^{G_K}$; we then have the following definitions:

\begin{definition}[\cite{Fontaine}]\label{Free Representations definition}
    \begin{enumerate}
        \item[(i)] A \emph{$B$-representation} $X$ of $G_K$ is a $B$-module of finite type equipped with a continuous action of $G_K$ satisfying:
        \begin{itemize}
            \item $g(x_1+x_2)=g(x_1)+g(x_2)$ for all $x_1,x_2\in X$ and $g\in G_K$;
            \item $g(\lambda x)=g(\lambda)g(x)$ for all $\lambda\in B, g\in G_K$ and $x\in X$.
        \end{itemize}
        \item[(ii)] A \emph{free $B$-representation} of $G_K$ is a $B$-representation such that the underlying $B$-module $X$ is free.
        \item[(iii)] A free $B$-representation $X$ of $G_K$ is \emph{trivial} if there exists a basis of $X$ consisting of elements fixed by $G_K$ (i.e., elements in $X^{G_K}$).
        \item[(iv)] Let $V$ be an $F$-representation of $G_K$. Then $V$ is called \emph{$B$-admissible} if $B\otimes_F V$ is a trivial $B$-representation of $G_K$.
    \end{enumerate}
\end{definition}

A fundamental result of the theory is that the category $\Rep^B_F(G_K)$ of $B$-admissible $F$-representations of $G_K$ forms a Tannakian subcategory of the category $\Rep_F(G_K)$ of all $F$-representations of $G_K$. Consequently, we can restrict our attention to such subcategories; different period rings thus yield different classes of representations. We aim to establish a similar framework in a multivariate setup. However, the classical definitions of $(F,G)$-regular rings given by Fontaine \cite{Fontaine} are not applicable in the multivariate case, as we must depart from the class of integral domains. To overcome this challenge, we utilise the properties of von Neumann regular rings and modules over them to formulate a definition suited for proving analogous results in the multivariate setup.\\

Let us briefly recall the core ideas of multivariate $p$-adic Hodge theory. For details, see foundational articles such as \cite{Zab, Zab2, Pal-Zab, CarKedZab, BCM}. Let $K_1, \dots, K_t$ be $p$-adic fields over a $p$-adic field $K_0$, and write $\G_{K,\Delta} = \G_{K_1} \times \dots \times \G_{K_t}$, which is a profinite group. Let $F$ be a $p$-adic field. The goal is to study the structure of the category $\Rep_F(\G_{K,\Delta})$ by examining its relevant Tannakian subcategories. These Tannakian subcategories arise from a commutative ring $B_\Delta$ equipped with a continuous action of $\G_{K,\Delta}$ such that $g(b_1+b_2) = g(b_1) + g(b_2)$ and $g(b_1b_2) = g(b_1)g(b_2)$ for all $g \in \G_{K,\Delta}$ and $b_1, b_2 \in B_\Delta$ and contains $F$ as a closed subfield.\\

Consider the perfectoid ring
\begin{equation*}
    \C_\Delta:=\C\hat{\otimes}_{K_0}\cdots\hat{\otimes}_{K_0}\C
\end{equation*}
where the completion is with respect to the semi-norm
\begin{equation*}
    |x|_\Delta:=\inf\left\{\max(|x_1|\cdots|x_t|)\;\middle|\; x=\sum x_1\otimes\cdots\otimes x_t\right\}.
\end{equation*}
This ring is an example of the perfectoid $\Delta$-fields considered in \cite{Pok-Pal}. It comes with a continuous action of $\G_{K,\Delta}$ given by the formula:
\begin{equation*}
  (g_1,\cdots,g_t)\left(\sum x_1\otimes\cdots\otimes x_t\right):=\sum g_1\cdot x_1\otimes\cdots\otimes g_t\cdot x_t.
\end{equation*}
In multivariate $p$-adic Hodge theory, $\C_\Delta$ is an analogue of $\C$. In \cite{BCM}, the authors consider $\C_\Delta$ with $K_0=\Frac W(k)$, the Witt vectors of a finite field $k$, and $\G_{K,\Delta}$ with $K_1=\cdots=K_t$. We, however, consider the general setup, which contains this diagonal setup as a particular case. Using the diagonal setup, \cite{BCM} defined the period rings $\mBdr, \BHT$, which are a robust construction based on results in commutative algebra, conditional on Proposition 4.3. In this article, we provide a construction of multivariate period rings that is more natural and avoids reliance on this Proposition.\\

The period rings defined in this article are slightly different from those of \cite{BCM} and can be related to them via profinite completion. However, this reformulation ensures that these period rings share many properties with the usual period rings, which is necessary for establishing the Tannakian structure. These new period rings are restricted to classical period rings when $|\Delta|=1$ and thus serve as a direct generalization of the classical theory.


\begin{subsection}{Acknowledgment}
The author is highly indebted to his advisor, Aprameyo Pal, for introducing him to these subjects and for many fruitful discussions. The author also thanks Kiran Sridhara Kedlaya for stimulating discussions during his visit to HRI. Financial support from the HRI Institute Fellowship for Research Scholars, the SERB MATRICS grant (MTR/2022/000302), and ANRF/ARG/2025/005885/MS is gratefully acknowledged.
\end{subsection}


\section{Preliminaries on Perfectoid $\Delta$-field $\C_\Delta$}

Let us recall some basic results given in \cite{Pok-Pal}. Let $s=\{s_\alpha\}_{\alpha\in\Delta}$ be a family of continuous embeddings $s_\alpha:\C\to \C$. Define an equivalence relation as the following: Say $s$ is equivalent to $s'$ if and only if there exists a continuous isomorphism $u:\C\to \C$ independent of $\alpha$ such that the diagram
\[
\begin{tikzcd}
    \C\arrow[rr,"u"] &  &  \C\\
        &\arrow[ul,"s_\alpha"] \C \arrow[ru,swap,"s'_\alpha"]&
\end{tikzcd}
\]
commutes. We denote its equivalence class by $\mathfrak{S}$ and call its element an \emph{embedding family}. Similarly, we have \emph{residual embedding family} denoted by $\mathfrak{K}$ (c.f. \cite{Pok-Pal}).\\

Suppose $s=\{s_\alpha\}_{\alpha\in\Delta}\in\mathfrak{S}$ be an embedding family, we have a surjective ring homomorphism 
\begin{align*}
    \beta_s: \C_\Delta&\to \C,\\
             \sum x_1\otimes\cdots\otimes x_t&\mapsto \sum s_1(x_1)\cdots s_t(x_t).
\end{align*}
The kernel $\mathfrak{P}_s$ is a maximal ideal of $\C_\Delta$. Then we have the following classification;

\begin{proposition}[\cite{Pok-Pal}, Proposition 2.4]
    There is a one-to-one correspondence
    \begin{align*}
        \mathfrak{S}&\to \spec\C_\Delta,\\
                  s &\mapsto \mathfrak{P}_s.
    \end{align*}
    In particular, every prime ideal of $\C_\Delta$ is maximal.
\end{proposition}

\begin{remark}
    This result is true because we have taken the completion of the tensor product with respect to the metric induced by the semi-norm $|\bullet|_\Delta$. For instance, if we take $p$-adic completion as in \cite{BCM}, then the ring has infinite Krull dimension.
\end{remark}

\begin{proposition}[\cite{Pok-Pal}, Proposition 2.6]
    $\C_\Delta$ is a reduced ring.
\end{proposition}

Therefore, $\C_\Delta$ is a reduced ring of dimension zero and hence a Von Neumann regular (VNR) ring (c.f. \cite{Pok-Pal} Section 4). An important property of the VNR ring is that every localisation at a prime ideal is a field. We have the following proposition.

\begin{proposition}[\cite{Pok-Pal},Lemma 4.3]
    The ring homomorphism $\beta_s:\C_\Delta\to \C$ induces an isomorphism from $(\C_\Delta)_{\mathfrak{P}_s}$ to $\C$, for every $s\in\mathfrak{S}$.
\end{proposition}

We have the analogue of the ring of integers for $\C_\Delta$, namely
\begin{align*}
  \mathcal{O}_{\C,\Delta}:=\mathcal{O}_{\C}\hat{\otimes}_{\mathcal{O}_{K_0}}\cdots\hat{\otimes}_{\mathcal{O}_{K_0}}\mathcal{O}_{\C}.
\end{align*}

The topology on $\C_\Delta$ is equivalent to the topology induced by the power-multiplicative semi-norm (c.f. \cite{Pok-Pal})
\begin{align*}
    |x|:=\sup\limits_{s\in\mathfrak{S}}|\beta_s(x)|, \text{ for } x\in \C_\Delta.
\end{align*}

Let $s=(s_\alpha)_{\alpha\in\Delta}$, then consider $K_s:=s_1(K_1)\cdots s_t(K_t)$. Then for every $s=(s_\alpha)_{\alpha\in\Delta}$, we have the following continuous group homomorphism
\begin{align*}
    \eta_s: \G_{K_s}&\to \G_{K,\Delta}\\
                  g&\mapsto (s^{-1}_1\circ g\circ s_1,\cdots, s^{-1}_t\circ g\circ s_t).
\end{align*}

We first prove a basic fact of the VNR ring, which will be used throughout the article:

\begin{lemma}\label{fixedpointofVNR}
    If a profinite group $G$ acts continuously on a topological VNR ring $R$, then the fixed point subring $R^G$ is also a VNR ring.
\end{lemma}

\begin{proof}
    To prove this lemma, we recall the following standard result from the general theory of VNR rings \cite{Rap}, Lemma 4:
    
    \begin{center}
    \emph{If $R$ is a VNR ring, then for every $x \in R$, there exists a unique $y \in R$ such that $x^2y = x$ and $y^2x = y$.}
    \end{center}
    
    We first demonstrate the existence of such a $y$. Since $R$ is a VNR ring, there exists $z \in R$ such that $x^2z = x$. Let $y = z^2x$. Then we have:
    \begin{align*}
        x^2y &= x^2(z^2x) = (x^2z)zx = xzx = x, \text{ and} \\
        xy^2 &= x(z^2x)^2 = xz^4x^2 = (x^2z)(xz^3) = (x^2z)z^2 = xz^2 = y.
    \end{align*}
    Thus, such an element $y$ exists. To show uniqueness, suppose $y_1 \in R$ satisfies $xy_1^2 = y_1$ and $x^2y_1 = x$. Then:
    \begin{align*}
        y_1 = xy_1^2 = (x^2y)y_1^2 = y(x^2y_1)y_1 = yxy_1 = y(x^2y)y_1 = y^2(x^2y_1) = y^2x = y.
    \end{align*}
    This establishes the uniqueness of $y$. We refer to this element $y$ as the \emph{pseudo-inverse} of $x$.

    Now, returning to the main proof: let $x \in R^G$. Since $R$ is a VNR ring, there exists a unique $y \in R$ such that $x^2y = x$ and $xy^2 = y$. For any $g \in G$, applying the group action yields:
    \begin{align*}
        g(x^2y) = g(x) \implies x^2g(y) = x
        \text{ and } g(xy^2) = g(y) &\implies xg(y)^2 = g(y).
    \end{align*}
    It follows that for every $g \in G$, $g(y)$ is also a pseudo-inverse of $x$. By the uniqueness proven above, we must have $g(y) = y$ for all $g \in G$. Therefore, $y \in R^G$, which confirms that $R^G$ is a VNR ring.
\end{proof}

We now prove the compatibility of $\beta_s$ and $\eta_s$.
\begin{proposition}
    There is a canonical isomorphism between $(\C^{\G_{K,\Delta}}_\Delta)_{\mathfrak{P}_s\cap \C^{\G_{K,\Delta}}_\Delta}$ and $\C^{\G_{K_s}}$.
\end{proposition}

\begin{proof}
   Let us restrict the surjective ring homomorphism $\beta_s:\C_\Delta\to \C$ for $s=(s_\alpha)_{\alpha\in\Delta}$ to $\C^{\G_{K,\Delta}}_\Delta$. Note that for $x=\sum x_1\otimes\cdots\otimes x_t\in\C_\Delta$ and $g\in\G_{K_s}$, we have 
   \begin{align*}
       g\cdot \beta_s(x)&=g\sum s_1(x_1)\cdots s_t(x_t)
                        =\sum g\circ s_1(x_1)\cdots g\circ s_t(x_t)\\
                        &=\sum s_1(s^{-1}_1\circ g\circ s_1)(x_1)\cdots s_t(s^{-1}_t\circ g\circ s_t)(x_1)
                        =\beta_s(\eta_s(g)\cdot x).
   \end{align*}
   This shows that if $x\in \C^{\G_{K,\Delta}}_\Delta$, then $g\cdot\beta_s(x)=\beta_s(x)$ i.e., $\beta_s(x)\in \C^{\G_{K_s}}$. Therefore, we have a ring homomorphism $\beta_s:\C^{\G_{K,\Delta}}_\Delta\to \C^{\G_{K_s}}$. Next, we claim that this map is surjective. Using Proposition 4.9 \cite{Fontaine}, we have $\C^{\G_{K_s}}=K_s$ and for $s_1(x_1)\cdots s_t(x_t)$, we have $x= x_1\otimes\cdots\otimes x_t\in\C^{\G_{K,\Delta}}_\Delta$ such that $\beta_s(x)=s_1(x_1)\cdots s_t(x_t)$, and hence the map is surjective. The kernel of this map is $\mathfrak{P}_s\cap \C^{\G_{K,\Delta}}_\Delta$. Consider the surjective ring homomorphism
   \begin{align*}
       (\C^{\G_{K,\Delta}}_\Delta)_{\mathfrak{P}_s\cap \C^{\G_{K,\Delta}}_\Delta}&\to \C^{\G_{K_s}},\\
                    \frac{a}{b} &\mapsto \frac{\beta_s(a)}{\beta_s(b)}.
   \end{align*}
   By Lemma \ref{fixedpointofVNR}, $\C^{\G_{K,\Delta}}_\Delta$ is a VNR ring. Thus, the left-hand side is a field, and hence the above map is an isomorphism. 
\end{proof}

It is also easy to see that every prime ideal of $\C^{\G_{K,\Delta}}_\Delta$ appears as the kernel of the ring homomorphism $\beta_s:\C^{\G_{K,\Delta}}_\Delta\to K_s$. \\

Since $\C_\Delta$ is a perfectoid $\Delta$-field, we can define its tilt as follows;

\begin{definition}
    The \emph{tilt} of $\C_\Delta$ is defined as the total ring of fractions of
    \begin{align*}
        \mathcal{O}^\flat_{\C,\Delta}:=\varprojlim\limits_{\Phi}(\mathcal{O}_{\C,\Delta}/(p)):=\left\{(x_n)_{n\in\mathbb{N}}\in\prod\limits_{n=0}^\infty(\mathcal{O}_{\C,\Delta}/(p))\;\middle|\; x^p_{n+1}=x_n\right\}.
    \end{align*}
    It is also denoted by $\C^\flat_\Delta.$
\end{definition}

The following proposition is central to the construction of multivariate period rings.

\begin{proposition}[Fontaine, Scholze]
    There exists a set-theoretic bijection 
    \begin{align*}
        \mathcal{O}^\flat_{\C,\Delta}\to \varprojlim\limits_{x\mapsto x^p}\mathcal{O}_{\C,\Delta}:=\left\{(x^{(n)})_{n\in\mathbb{N}}\in\prod\limits_{n=0}^\infty\mathcal{O}_{\C,\Delta}\;\middle|\;(x^{(n+1)})^p=x^{(n)}\right\}.
    \end{align*}
\end{proposition}
Transporting the structure, we can define a ring structure on $\varprojlim\limits_{x\mapsto x^p}\mathcal{O}_{\C,\Delta}$ as follows; 
\begin{align*}
    (x+y)^{(n)}&=\lim\limits_{m\to \infty}(x^{(m+n)}+y^{(m+n)})^{p^m}\\
    (xy)^{(n)} &=x^{(n)}y^{(n)}
\end{align*}

Using the proposition, we also have a surjective map
\begin{align*}
    \#: \mathcal{O}^\flat_{\C,\Delta}\to \mathcal{O}_{\C,\Delta}
\end{align*}
which extends to a map from $\C^\flat_\Delta$ to $\C_\Delta$.\\

The VNR ring $\C^\flat_\Delta$ is also a perfectoid $\Delta$-field of characteristics $p$, which is complete with respect to the power-multiplicative semi-norm
\begin{align*}
    |x|_\flat=|x^\#|.
\end{align*}
We have the following proposition;

\begin{proposition}[\cite{Pok-Pal}]
    There exists a one-to-one correspondence
    \begin{align*}
        \spec\C^\flat_\Delta&\to \mathfrak{S},\\
        \text{and }\spec\mathcal{O}^\flat_{\C,\Delta}\setminus\spm\mathcal{O}^\flat_{\C,\Delta}&\to \mathfrak{S},\\
      \text{and }  \spm\mathcal{O}^\flat_{\C,\Delta}&\to \mathfrak{K}.
    \end{align*}
\end{proposition}

We note that $\G_{K,\Delta}$ acts continuously on $\C^\flat_\Delta$ by the formula $g\cdot(x^{(n)})=(gx^{(n)})_{n\in\mathbb{N}}$ for $g\in\G_{K,\Delta}$ and $(x^{(n)})_{n\in\mathbb{N}}\in\C^\flat_\Delta$.

\section{The de Rham period rings}

We now define the analogue of the Fontaine period ring $\Ainf$. This definition is not the same as given in \cite{BCM}, but this is more closely related to classical setups.

\begin{definition}
    The \emph{Fontaine period ring} $\mAinf$ is defined as
    \begin{align*}
        \mAinf:=W(\mathcal{O}^\flat_{\C,\Delta}).
    \end{align*}
\end{definition}

\begin{remark}
    \begin{enumerate}
        \item[(i)]  Every element of $\mAinf$ has a T\'eichmuller expansion: that is for every $x\in\mAinf$, there exists $(c_n)_{n\in\mathbb{N}}\in\mathcal{O}^\flat_{\C,\Delta}$ such that
        \begin{align*}
          x=[c_0]+[c_1]p+[c_2]p^2+\cdots=\sum\limits_{k=0}^\infty[c_k]p^k.
        \end{align*}
        \item[(ii)] $\G_{K,\Delta}$ acts naturally on $\mAinf$ by the formula $g\cdot x=\sum\limits_{k=0}^\infty[g\cdot c_k]p^k$.
    \end{enumerate}
   
\end{remark}

\begin{definition}
    The \emph{Fontaine theta map} is defined as the homomorphism
    \begin{align*}
      \theta_\Delta:  \mAinf&\to \mathcal{O}_{\C,\Delta}\\
        \sum\limits_{i=0}^\infty c_n p^n&\mapsto \sum\limits_{i=0}^\infty c^\#p^n
    \end{align*}
\end{definition}

\begin{remark}
    \begin{enumerate}
        \item[(i)] This map makes sense because $\mathcal{O}_{\C,\Delta}$ is $p$-adically complete.
        \item[(ii)] Note that going modulo $p$, the above map reduces to the ring homomorphism $\#:\mathcal{O}^\flat_{\C,\Delta}\to \mathcal{O}_{\C,\Delta}/(p)$. In fact, this map is defined through the functoriality of Witt vectors.
        \item[(iii)] $\theta_\Delta$ is $\G_{K,\Delta}-$equivariant. Indeed, for $g\in\G_{K,\Delta}$ and $x=\sum\limits_{i=0}^\infty[x_i]p^i\in\mAinf$, we get
      \begin{align*}
          \theta_\Delta(gx)=\sum\limits_{i=0}^\infty g(x^\#_i)p^i=g\sum\limits_{i=0}^\infty x^\#_ip^i=g\theta_\Delta(x).
      \end{align*}
    \end{enumerate}
\end{remark}

Note that the ring homomorphism $\theta_\Delta$ is surjective. Therefore, we want to find the kernel.\\

Choose an element $\varpi\in \mathcal{O}^\flat_{\C,\Delta}$ such that $\varpi^\#=-p$, which exists because of surjectivity of $\#$. We know that $|\varpi|_s=|p|=\frac{1}{p}$ for every $s\in \mathfrak{S}$, and so if $\varpi'^\#=-p$, then $\varpi$ and $\varpi'$ differs by an unit. Set 
\begin{align*}
    \xi=[\varpi]+p=(\varpi,1,0,\cdots)\in\mAinf,
\end{align*}
then $\theta_\Delta(\xi)=\varpi^\#+p=0$. Therefore $\xi\in\ker\theta_\Delta$.

\begin{proposition}
    The kernel $\ker\theta_\Delta$ is a principal ideal in $\mAinf$ generated by $\xi$.
\end{proposition}

\begin{proof}
    We claim that it is enough to show $\ker\theta_\Delta\subset(\xi,p)$. Indeed, if $x\in\ker\theta_\Delta$ satisfies $x=\xi y_0+px_1$, then $\theta_\Delta(x)=0=p\theta_\Delta(x_1)$. This gives $x_1\in\ker\theta_\Delta$. Thus $x_1=\xi y_1+px_2$ showing $x_2\in\ker\theta_\Delta$. By induction, for every $n\in\mathbb{N}$, there exists $x_n\in\ker\theta_\Delta$ such that $x=\xi y_0+p\xi y_1+\cdots+p^{n-1}\xi y_{n-1}+p^n x_n$. But $\mathcal{O}_{\C,\Delta}$ is $p-$adically seperated and complete, we must have $x\in (\xi)$.

    Now assume $x=\sum\limits_{i=0}^\infty [x_i]p^i$, then $\theta_\Delta(x)=\sum\limits_{i=0}^\infty x^\#_ip^i=0$. This means that $-p\mid x^\#_0$ and so $\varpi\mid x_0$. Hence there exists $b_0\in\mathcal{O}^\flat_{\C,\Delta}$ such that $x_0=b_0\varpi$. Let $b=[b_0]$, then 
    \begin{align*}
        x-b\xi&=\sum\limits_{i=0}^\infty [x_i]p^i-[b_0]([\varpi]+p)
              =[x_0]+\sum\limits_{i=1}^\infty[x_i]p^i-[b_0][\varpi]+p[b_0]
              =p\cdot *\in p\mAinf.
    \end{align*}
    Therefore, $x\in (\xi,p)$, which we wanted to show.

    Suppose $x\in \bigcap\limits_{i=1}^\infty(\ker\theta_\Delta)^n=\bigcap\limits_{i=1}^\infty(\xi^n)$, then for all $n$, we have $x=\xi^n y_n$ for some $y_n$. This shows $x^\#=-p^ny^\#_n$ for every $n=1,2,\cdots$, proving $x^\#=0$ in $\mathcal{O}_{\C,\Delta}$. Hence $x=0$.
\end{proof}

\begin{remark}
    \begin{enumerate}
        \item[(i)]  In \cite{BCM}, the authors proved that $\ker\theta_\Delta$ is generated by $t$-many elements, but here we are saying they differ by units.
    \end{enumerate}
\end{remark}

Since $\mathcal{O}^\flat_{\C,\Delta}$ is a perfect ring of characteristics $p$, the ring $\mAinf$ is $p-$torsion free, $p$-adically complete ring with $\mAinf/(p)\cong \mathcal{O}^\flat_{\C,\Delta}$. Thus the ring homomorphism $\theta_\Delta$ extended to a surjective $\G_{K,\Delta}$ equivariant surjective ring homomorphism  
\begin{align*}
   \theta_\Delta: \mAinf\left[\frac{1}{p}\right]&\to {\C}_\Delta\\
   \sum\limits_{i\geq n}[x_i]p^i &\mapsto \sum\limits_{i\geq n} x^\#p^i
\end{align*}
whose kernel is still generated by $\xi$. The following definition of de Rham period ring is closely related to one given in \cite{BCM}:

\begin{definition}
    \begin{enumerate}
        \item[(i)]\footnote{Here, $\bullet^{(p)}$ stands for pre, this is because we will see later that these period rings are not well behaved and have a different ring theoretical structure than the de Rham period ring in classical $p$-adic Hodge Theory.} The \emph{integral multivariate pre-de Rham period ring} $\pmBdrp$  is defined as
        \begin{align*}
            \pmBdrp:=\varprojlim\limits_{n\in\mathbb{N}}\mAinf\left[\frac{1}{p}\right]/(\xi^n)
        \end{align*}
        \item[(ii)] The \emph{multivariate pre-de Rham period ring} $\pmBdr$ is defined as
        \begin{align*}
            \pmBdr:=\pmBdrp\left[\frac{1}{\xi}\right].
        \end{align*}
    \end{enumerate}
\end{definition}

Let $s\in\mathfrak{S}$ be a family of embeddings, then we have the induced ring homomorphism

\begin{align*}
    \mAinf\left[\frac{1}{p}\right]&\to \Ainf\left[\frac{1}{p}\right].\\
     \sum\limits_{i\geq n}[x_i]p^i&\mapsto \sum\limits_{i\geq n}[\beta^\flat_s(x_i)]p^i
\end{align*}

Note that $\xi\in\mAinf$ is mapped to a generator of $\ker\theta$ of the usual $p$-adic Hodge Theory of \cite{Fontaine}. Therefore, the above map generates a surjective ring homomorphism 
\begin{align*}
    \Phi_s:\pmBdrp\to \Bdrp.
\end{align*}

Recall that $\Bdrp$ is a complete discrete valuation ring whose residue field is $\C$ and whose field of fractions is $\Bdr$. Note that $\xi_s=\Phi_s(\xi)$ satisfies $\xi_s=[\beta^\flat_s(\varpi)]+p\in\Ainf$, and hence generates the kernel of the usual Fontaine's $\theta-$map. The ring $\Bdr$ is equipped with a continuous action of $\G_K$ for any $p$-adic field $K$, and in particular, of $\G_{K_s}$ for $K_s=s_1(K_1)\cdots s_t(K_t)$. Note that we also have an injective homomorphism $\G_{K_s}\to \G_{K,\Delta}$ given by $g\mapsto (s^{-1}_1\circ g\circ s_1,\cdots,s^{-1}_t\circ g\circ s_t)$. These morphisms, $K_\Delta\to K_s$ and $\G_{K_s}\to \G_{K,\Delta}$, are compatible with the corresponding actions on $\C_\Delta$. That is, for every $g\in \G_{K_s}$ and $x_1\otimes\cdots\otimes x_t\in C_\Delta$, we have
\begin{align*}
   \beta_s((s^{-1}_1\circ g\circ s_1,\cdots,s^{-1}_t\circ g\circ s_t)(x_1\otimes\cdots\otimes x_t))=g(s_1 x_1)\cdots g(s_t x_t)=g\beta_s(x_1\cdots x_t).
\end{align*}
Similarly, this injection is compatible with the actions defined for $\mathcal{O}_{\C,\Delta}, \mBdrp, \mBdr$ and corresponding rings in single-variable setups, for each $s\in\mathfrak{S}$.\\

Also, consider the surjective ring homomorphism
\begin{align*}
    \overline{\Phi}_s:\pmBdrp\to \C
\end{align*}
by composing the ring homomorphism $\Phi_s$ with the canonical surjection $\theta:\Bdrp\to \C$. Furthermore, $\theta_\Delta$ induces a surjective ring homomorphism 
\begin{align*}
    \theta_\Delta:\pmBdrp\to\C_\Delta
\end{align*}
with kernel $\xi\pmBdrp$ such that the diagram
\[
\begin{tikzcd}
    \pmBdrp \arrow[r,"\Phi_s"]\arrow[d,swap,"\theta_\Delta"]& \Bdrp\arrow[d,"\theta"] \\
    \C_\Delta \arrow[r,"\beta_s"]     & \C        
\end{tikzcd}
\]
commutes in the category of rings. Indeed, if $x=\sum\limits_{i\geq n}[x_i]p^i\in\pmBdrp$ has image $\sum\limits_{i\geq n}\beta^\flat_s(x_i)^\#p^i$ under clockwise composition and has image $\sum\limits_{i\geq n}\beta_s(x^\#_i)p^i$ under anticlockwise composition, which agrees. Thus, we have the following proposition

\begin{proposition}\label{Prime Ideal of de Rham period ring}
    We have a one-to-one correspondence between 
    \begin{align*}
        \mathfrak{S}\leftrightarrow\spm\pmBdrp.
    \end{align*}
\end{proposition}

\begin{proof}
    For $s\in\mathfrak{S}$, the kernel $\ker\overline{\Phi}_s$ is a maximal ideals of $\pmBdrp$ and $\ker\Phi_s$ is a prime ideal of $\pmBdrp$ which is not maximal. The prime ideal $\ker\Phi_s$ does not contains the kernel $\ker\theta_\Delta=(\xi)$, but $\ker\overline{\Phi}_s$ does. Suppose $\mathfrak{Q}$ be a prime ideal which contains $\ker\theta_\Delta$, then $\theta_\Delta(\mathfrak{Q})$ is a prime ideal of $\C_\Delta$ and thus is equal to either $(0)$ or $\mathfrak{P}_s$ for some $s\in\mathfrak{S}$. Since $(\xi)$ is in the Jacobson radical of $\pmBdrp$, any maximal ideal contains $\ker\theta_\Delta$. If $\mathfrak{Q}\in\spm\pmBdrp$, then we claim that $\mathfrak{Q}=\ker\overline{\Phi}_s$. Note that $\ker\overline\Phi_s=\ker\Phi_s+(\xi)$. If $x\in\ker\Phi_s$ then 
    \begin{align*}
       \beta_s\circ\theta_\Delta(x)=\theta\circ\Phi_s(x)=0\implies \theta_\Delta(x)\in\mathfrak{P}_s=\theta_\Delta(\mathfrak{Q})
    \end{align*}
    Since if $\mathfrak{Q}$ is maximal, then $\theta_\Delta(\mathfrak{Q})$ is a maximal ideal in $\C_\Delta$. Thus we have $\theta_\Delta(x-q)=0$ for some $q\in\mathfrak{Q}$. In this case $x-q\in \ker\theta_\Delta\subset\mathfrak{Q}$. Therefore, $\ker\overline{\Phi}_s=\ker\Phi_s+\ker\theta_\Delta\subset\mathfrak{Q}$. But $\ker\overline{\Phi}_s$ is maximal, therefore we have $\mathfrak{Q}=\ker\overline\Phi_s$. Note that if for $s.t\in\mathfrak{S}$, we have $\ker\overline\Phi_s=\ker\overline\Phi_t$, then we have $s=t$.
\end{proof}

\begin{remark}\label{Prime ideals containing the kernel}
    If $\mathfrak{Q}$ is a prime ideal containing $\xi^n,\;n\in\mathbb{N}^*$, then $\xi\in\mathfrak{Q}$. This shows that this prime ideal contains $\ker\theta_\Delta$. The above proof also shows that maximal prime ideals that do not contain $\ker\theta_\Delta$ are of the form $\ker\Phi_s$.
\end{remark}

\begin{corollary}\label{Prime ideal of prim-integral period ring}
    We have a one-to-one correspondence between
    \begin{align*}
        \mathfrak{S}\leftrightarrow\spm\pmBdr.
    \end{align*}
    Moreover, the intersection of maximal ideals is $\{0\}$.
\end{corollary}

\begin{proof}
   Since we have $\pmBdr=\pmBdrp\left[\frac{1}{\xi}\right]$ and the maximal prime ideals which do not contains $\xi$ are of the form $\ker\Phi_s$, the localization $\pmBdr$ has the maximal ideals equals to extension $\Phi_s:\pmBdr\to \Bdr$ of $\Phi_s:\pmBdrp\to\Bdrp$. Note that by definition, the intersection of the kernels of the ring homomorphisms $\mAinf\to\Ainf$ corresponding to $s\in\mathfrak{S}$ is equal to zero, $\bigcap\limits_{s\in\mathfrak{S}}\ker\Phi_s=\{0\}$, which we wanted to prove.
\end{proof}

We now recall the following definition from classical ring theory.

\begin{definition}\label{Profinite completion}
    Let $R$ be a ring and let $\{\mathfrak{m}_i\}_{i\in\Lambda}$ be a family of maximal ideals of $R$ such that $\bigcap\limits_{i\in\Lambda}\mathfrak{m}_\lambda=\{0\}$. Let $\Lambda'\subset\Lambda$ be a finite subset, and define $\mathfrak{a}_{\Lambda'}=\bigcap\limits_{i\in\Lambda'}\mathfrak{m}_i$, then the \emph{profinite completion} of $R$ with respect to the family $\{\mathfrak{m}_i\}_{i\in\Lambda}$ is defined to be
    \begin{align*}
        \hat{R}:=\varprojlim\limits_{\Lambda'\subset\Lambda\text{ finite }}R/\mathfrak{a}_{\Lambda'}.
    \end{align*}
\end{definition}

Note that there is a monomorphism $R\hookrightarrow\hat{R}$ because the kernel is $\bigcap\limits_{\Lambda'\subset\Lambda\text{ finite }}\mathfrak{a}_{\Lambda'}=\bigcap\limits_{i\in\Lambda}\mathfrak{m}_i=\{0\}$. Moreover, $\hat{R}$ is a VNR ring with maximal ideals given by $\mathfrak{M}=\varprojlim\limits_{\Lambda'\subset\Lambda\text{ finite }}\mathfrak{m}/\mathfrak{a}_{\Lambda'}$, whose residue field is isomorphic to $R/\mathfrak{m}$. This ring has a natural topology where the fundamental system of neighborhoods of 0 is given by $\bigcap\limits_{i\in\Lambda'}\mathfrak{M}_i$ with $\Lambda'\subset\Lambda$ finite, which we call the \emph{profinite topology}.\\

Using this, we define the multivariate de Rham period ring:

\begin{definition}
    The \emph{multivariate de Rham period ring} $\mBdr$ is defined to be the profinite completion of $\pmBdr$ with respect to the family of ideals $\{\ker\Phi_s\}_{s\in\mathfrak{S}}$. The \emph{integral $\Delta$-de Rham period ring} is the closure of $\pmBdrp$ inside $\mBdr$.
\end{definition}

\begin{proposition}
    There exists a one-to-one correspondence 
    \begin{align*}
        \mathfrak{S}&\leftrightarrow\spec\mBdr,\\
            s       &\mapsto \ker\Phi_s.
    \end{align*}
    In particular, $\mBdr$ is a VNR ring. There exists a one-to-one correspondence
    \begin{align*}
          \mathfrak{S}&\leftrightarrow \spec\mBdrp\setminus \spm\mBdrp\\
                    s &\mapsto \ker\Phi_s\\
       \text{and }\;\mathfrak{S}&\leftrightarrow\spm\mBdrp\\
                              s &\mapsto \ker\Phi_s+(\xi).
    \end{align*}
    In particular, $\mBdrp$ is of Krull dimension 1.
\end{proposition}

\begin{proof}
    The first part is immediate from Proposition \ref{Prime Ideal of de Rham period ring} and Corollary \ref{Prime ideal of prim-integral period ring}. The second part follows from Remark \ref{Prime ideals containing the kernel}.
\end{proof}

\begin{remark}
    It is easy to see that $\mBdrp$ can be constructed from $\pmBdrp$ by an analogous construction given in Definition \ref{Profinite completion} with respect to the family of ideals $\Big\{\ker\Phi_s:\pmBdrp\to\Bdrp\Big\}_{s\in\mathfrak{S}}$, whose intersection is also zero.
\end{remark}

\begin{proposition}
    The map $\theta_\Delta:\pmBdrp\to\C_\Delta$ has an extension to a surjective ring homomorphism $\theta_\Delta:\mBdrp\to\widehat{\C}_\Delta$, where $\widehat{\C}_\Delta$ is the profinite completion of $\C_\Delta$ with respect to the family $\{\mathfrak{P}_s\}_{s\in\mathfrak{S}}$ in sense of Definition \ref{Profinite completion}.
\end{proposition}

\begin{proof}
    We have the commutative diagram
    \[
    \begin{tikzcd}
        \pmBdrp\arrow[r,"\theta_\Delta"]\arrow[d,"\Phi_s"]& \C_\Delta\arrow[d,"\beta_s"]\\
        \Bdrp\arrow[r,"\theta"]  &\C
    \end{tikzcd}
    \]
    i.e., $\beta_s\circ\theta_\Delta=\theta\circ\Phi_s$. Then we have
    \begin{align*}
        \theta^{-1}_\Delta(\mathfrak{P}_s)=\ker(\beta_s\circ\theta_\Delta)=\ker(\theta\circ\Phi_s)=\Phi^{-1}_s(\ker\theta)\supset \ker\Phi_s.
    \end{align*}
    Then for any finite subset $\mathfrak{S}'$ of $\mathfrak{S}$, define a ring homomorphism
    \begin{align*}
        \pmBdrp/\bigcap\limits_{s\in\mathfrak{S}'}\ker\Phi_s&\to \C_\Delta/\bigcap\limits_{s\in\mathfrak{S}'}\mathfrak{P}_s,\\
        x+\bigcap\limits_{s\in\mathfrak{S}'}\ker\Phi_s&\mapsto \theta_\Delta(x)+\bigcap\limits_{s\in\mathfrak{S}'}\mathfrak{P}_s.
    \end{align*}
    We want to show that this map is well-defined. For this, suppose that $x+\bigcap\limits_{s\in\mathfrak{S}'}\ker\Phi_s=y+\bigcap\limits_{s\in\mathfrak{S}'}\ker\Phi_s$. Then we have
    \begin{align*}
        x-y\in\ker\Phi_s\text{ for all }s\in\mathfrak{S}'&\implies x-y\in\ker\Phi_s\text{ for all }s\in\mathfrak{S}'\\
        &\implies x-y\in\theta^{-1}_\Delta(\mathfrak{P}_s)\text{ for all }s\in\mathfrak{S}'\\
        &\implies \theta_\Delta(x)-\theta_\Delta(y)\in\mathfrak{P}_s\text{ for all }s\in\mathfrak{S'}.
    \end{align*}
    This shows that $\theta_\Delta(x)+\bigcap\limits_{s\in\mathfrak{S}'}\mathfrak{P}_s=\theta_\Delta(y)+\bigcap\limits_{s\in\mathfrak{S}'}\mathfrak{P}_s$. This shows that the above map is well-defined, and it is clear that this extends the ring map $\theta_\Delta:\pmBdrp\to\C_\Delta$, and taking the inverse limit we get the required ring homomorphism. Also note that $\ker\theta_\Delta$ is principal, which is generated by the system of images of $\xi$ onto each quotient, which we also denote by $\xi$. We denote this map also by $\theta_\Delta$.
\end{proof}

\begin{definition}
    For each $i\in\mathbb{Z}$, we define $\Fil^i_\Delta\mBdr=\xi^i\mBdrp$.
\end{definition}

\begin{remark}
    For every $i \in \mathbb{Z}$, we have the relation $\Fil^i \Bdr = \Phi_s(\Fil^i_\Delta \mBdr)$, where $\Fil^i \Bdr$ denotes the standard filtration on $\Bdr$.
\end{remark}
    
\begin{remark}
    The filtration forms a descending chain
    \begin{align*}
        \cdots \supseteq \Fil^{-1}_\Delta \mBdr \supseteq \mBdrp \supseteq \Fil^1_\Delta \mBdr \supseteq \cdots,
    \end{align*}
    which satisfies $\bigcap_{i \in \mathbb{Z}} \Fil^i_\Delta \mBdr = \{0\}$ and $\bigcup_{i \in \mathbb{Z}} \Fil^i_\Delta \mBdr = \mBdr$. This induces a corresponding filtration 
    \begin{align*}
        \cdots \supseteq \Fil^{-1} \Bdr \supseteq \Bdrp \supseteq \Fil^1 \Bdr \supseteq \cdots
    \end{align*}
    that is independent of $s \in \mathfrak{S}$. Indeed, each $\xi_s = \Phi_s(\xi)$ generates the standard $\ker \theta$, and $\xi_s$ differs from the corresponding $\xi_t$ for $t \in \mathfrak{S}$ only by a unit.
\end{remark}

\begin{remark}
    As in \cite{Fontaine}, there are two natural topologies on $\mBdr$:
    \begin{itemize}
        \item The topology defined by declaring the collection $\{\Fil^i_\Delta \mBdrp\}_{i \in \mathbb{Z}}$ as a fundamental system of neighborhoods of $0$.
        \item The \emph{canonical topology}, induced by the inverse limit, where each component $\mAinf[\frac{1}{p}] / (\ker \theta_\Delta)^n$ is equipped with the quotient topology from $\mAinf[\frac{1}{p}]$. The topology generated by this and the pro-finite completion \ref{Profinite completion} gives a topology on $\mBdr$.
    \end{itemize}

    The standard tensor product $\overline{K}_0 \otimes_{K_0} \cdots \otimes_{K_0} \overline{K}_0$ is a subring of $\mBdrp$, and the inclusion respects the action of $\G_{K,\Delta}$. Indeed, the ring homomorphism $\iota_\alpha: \C \to \C_\Delta$ induces a ring homomorphism $\Phi_\alpha: \Bdrp \to \mBdrp$ with a group action compatible with the projection map $\G_{K,\Delta} \to \G_{K_\alpha}$. This follows because we can choose $\varpi \in K_\alpha$ such that $\varpi^\# = -p$; thus, the map preserves the filtration in the sense that $\Phi_\alpha(\Fil^n \Bdr) \subseteq \Fil^n_\Delta \mBdr$. \\

    We claim that $\Phi_\alpha$ is injective. Since this map is induced at the level of $\Ainf$, we first restrict our attention to $\Phi_\alpha: \Ainf \to \mAinf$. Let $x = \sum\limits_{n=0}^\infty [x_{\alpha,n}] p^n \in \Ainf$ be such that $\Phi_\alpha(x) = 0$. Then $\varpi_\alpha$ divides $x_{\alpha,0}$, which implies that $\varpi_\alpha \in \C^\flat_\Delta$ divides $\iota^\flat_\alpha(x_\alpha)$, and consequently $x = px_1$. Iterating this argument, we obtain $x = 0$. \\

    Since $\Phi_s$ respects $\ker \theta$, we obtain an injective ring homomorphism $\Bdrp \to \mBdrp$. If $K_0 = \Frac W(k)$, then $W(k)$ remains stable throughout the construction, and $\Phi_\alpha$ is a $K_0$-algebra homomorphism. By Lemma 6.17 of \cite{Fontaine} and the universal property of tensor products, we have $\overline{K}_1 \otimes_{K_0} \cdots \otimes_{K_0} \overline{K}_t \subset \mBdrp$ and $(\overline{K}_1 \otimes_{K_0} \cdots \otimes_{K_0} \overline{K}_t) \cap \Fil^1_\Delta \mBdr = \{0\}$ (using the fact that $\bigcap_{s \in \mathfrak{S}} \ker \Phi_s = \{0\}$). For an arbitrary $K_0$, we consider the maximal unramified extension of $\mathbb{Q}_p$ inside $K_0$ and proceed similarly using ramified Witt vectors.
\end{remark}

Consider $\epsilon=(1,\zeta_p,\zeta_{p^2},\cdots)\in\mathcal{O}^\flat_{\C,\Delta}$ made from the compatible $p^n-$th roots of unity and define $\pi=[\epsilon]-1\in\mAinf$. Then $\theta_\Delta(\pi)=\epsilon^\#-1=0$, that is $\pi\in\ker\theta_\Delta=\Fil^1_\Delta\mBdr$. Thus $(-1)^{n+1}\frac{([\epsilon]-1)^n}{n}\in\pmBdrp\subset\mBdrp$.

\begin{definition}
    We define the element $t_\Delta$ as
    \begin{align*}
        t_\Delta:=\log[\epsilon]=\sum\limits_{n=1}^\infty(-1)^{n+1}\frac{([\epsilon]-1)^n}{n}\in\mBdrp.
    \end{align*}
\end{definition}

\begin{lemma}
    The element $t_\Delta\in\Fil^1_\Delta(\mBdr)\setminus\Fil^2_\Delta(\mBdrp)$.
\end{lemma}

\begin{proof}
    The proof is clear from Proposition 6.22 of \cite{Fontaine}, since $\Phi_s$ respects filtrations.
\end{proof}

Note that $\G_{K,\Delta}$ acts continuously on $\mBdrp$ and the natural map $\Phi_s:\mBdr\to\Bdr$ is compatible with the map $\eta_s:\G_{K_s}\to\G_{K,\Delta}$, which we will use in \S\ref{Admissible Representations}.\\

In the next section, we will define the Hodge-Tate period ring, which is related to $\mBdrp$ using this element $t_\Delta$.


\section{Hodge-Tate Period Rings}

To define the Hodge-Tate period ring, we first need to define the analogue of the Tate twist. Specifically, we wish to define the $p$-adic cyclotomic character in this setup. For this, we have an analogue of the cyclotomic tower in the multivariate case. This is represented as follows: 

{\small
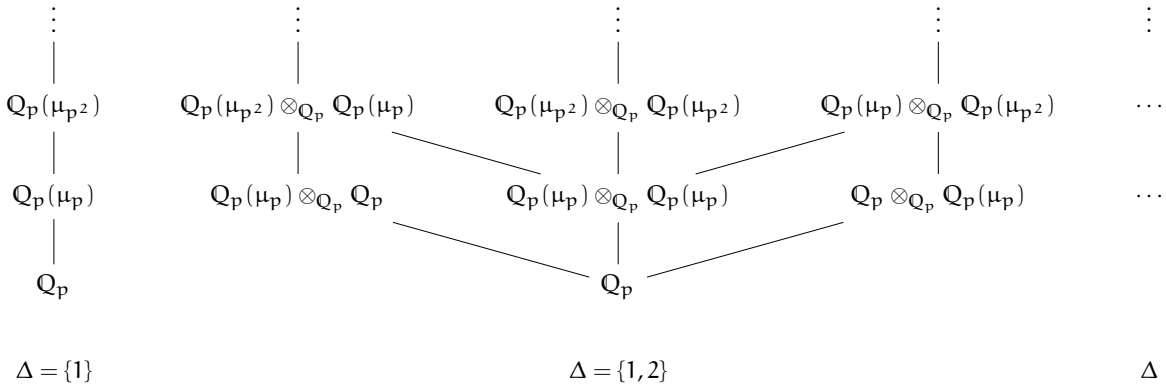
\[
\begin{tikzcd}
    \vdots    &  \vdots & \vdots & \vdots & \vdots\\
    \mathbb{Q}_p(\mu_{p^2})\arrow[u,dash]  & \mathbb{Q}_p(\mu_{p^2})\otimes_{\mathbb{Q}_p}\mathbb{Q}_p(\mu_{p})\arrow[u,dash] & \mathbb{Q}_p(\mu_{p^2})\otimes_{\mathbb{Q}_p}\mathbb{Q}_p(\mu_{p^2})\arrow[u,dash] & \mathbb{Q}_p(\mu_{p})\otimes_{\mathbb{Q}_p}\mathbb{Q}_p(\mu_{p^2})\arrow[u,dash] &\cdots\\
    \mathbb{Q}_p(\mu_{p})\arrow[u,dash]  & \mathbb{Q}_p(\mu_p)\otimes_{\mathbb{Q}_p}\mathbb{Q}_p \arrow[u,dash] & \mathbb{Q}_p(\mu_{p})\otimes_{\mathbb{Q}_p}\mathbb{Q}_p(\mu_{p})\arrow[ul,dash]\arrow[ur,dash]\arrow[u,dash]& \mathbb{Q}_p\otimes_{\mathbb{Q}_p}\mathbb{Q}_p(\mu_p) \arrow[u,dash] & \cdots\\
    \mathbb{Q}_p\arrow[u,dash] & & \mathbb{Q}_p\arrow[ul,dash]\arrow[u,dash]\arrow[ur,dash]\\
    \Delta=\{1\}  &  & \Delta=\{1,2\} & & \Delta
\end{tikzcd}
\]
}

Suppose $\mathbb{N}^t$ is equipped with the usual partial ordering; that is, for $\bn=(n_1,\dots,n_t)$ and $\bn'=(n'_1,\dots,n'_t)$ in $\mathbb{N}^t$, we define $\bn\leq \bn'$ if and only if $n_\alpha\leq n'_\alpha$ for every $\alpha \in \Delta=\{1,\dots,t\}$. We define a subgroup
\begin{align*}
    \mu_{\bp^{\bn}} := \left\{ \boldsymbol\zeta = \zeta_1 \otimes \dots \otimes \zeta_t \in \C^\times_\Delta \;\middle|\; \zeta_1^{p^{n_1}} = \dots = \zeta_t^{p^{n_t}} = 1 \text{ for } \bn=(n_1,\dots,n_t) \right\}.
\end{align*}
Note that $\G_{\mathbb{Q}_p,\Delta}$ acts naturally on $\mu_{\bp^{\bn}}$. Here, for $\bn=(n_1,\dots,n_t)$, we write $\bp^{\bn}=(p^{n_1},\dots,p^{n_t})$ and $\boldsymbol{\zeta}^{\bn}=\zeta_1^{n_1}\otimes\dots\otimes\zeta_t^{n_t}$. We call an element $\boldsymbol\zeta \in \mu_{\bp^{\bn}}$ \emph{primitive} if for every $\bn' \leq \bn$ in $\mathbb{N}^t$ such that $\bn' \neq \bn$, we have $\boldsymbol\zeta^{\bp^{\bn'}} \neq 1$. Let $\boldsymbol\zeta_{\bp^\bn}$ be a primitive element in $\mu_{\bp^{\bn}}$; then we have a canonical group isomorphism:
\begin{align*}
    (\mathbb{Z}/p^{n_1}\mathbb{Z}) \times \dots \times (\mathbb{Z}/{p^{n_t}\mathbb{Z}}) &\to \mu_{\bp^{\bn}} \\
    \bn' = (n'_1, \dots, n'_t) &\mapsto \boldsymbol{\zeta}^{\bn'}_{\bp^{\bn}}.
\end{align*}

Since $\G_{\mathbb{Q}_p,\Delta}$ acts simply transitively on the set of primitive roots in $\mu_{\bp^{\bn}}$, for every $\boldsymbol\sigma=(\sigma_1,\dots,\sigma_t) \in \G_{\mathbb{Q}_p,\Delta}$, there exists a unique $\ba(\boldsymbol\sigma,\bn) \leq \bn$ such that $\boldsymbol\sigma \cdot \boldsymbol\zeta_{\bp^\bn} = \boldsymbol\zeta^{\ba(\boldsymbol\sigma,\bn)}_{\bp^{\bn}}$. Note that the subset of primitive elements is isomorphic to $(\mathbb{Z}/p^{n_1}\mathbb{Z})^{\times} \times \dots \times (\mathbb{Z}/p^{n_t}\mathbb{Z})^\times$. Let $\boldsymbol{1}_\alpha$ denote the vector with $1$ in the $\alpha$-th entry and $0$ elsewhere, and let $\bp_\alpha = p \cdot \boldsymbol{1}_\alpha$. Then we have:
\begin{align*}
    \boldsymbol\zeta^{\bp_\alpha}_{\bp^{\bn+\boldsymbol{1}_\alpha}} = \boldsymbol\zeta_{\bp^\bn}
\end{align*}
Applying $\boldsymbol{\sigma}$ to both sides, we see that $\ba(\boldsymbol\sigma,\bn+\boldsymbol{1}_\alpha) \equiv \ba(\boldsymbol{\sigma},\bn) \pmod{\bp_\alpha}$. Consequently, taking the limit over $\bn$, we obtain a homomorphism:
\begin{align*}
   \boldsymbol{\chi}: \G_{\mathbb{Q}_p,\Delta} &\to \GL_t(\mathbb{Z}_p) \\
   \boldsymbol{\sigma} &\mapsto \diag(\ba(\boldsymbol\sigma,\bn))
\end{align*}
where for $\textbf{m}=(m_1,\dots,m_t)$, the symbol $\diag(\textbf{m})$ denotes the diagonal matrix:
\begin{align*}
    \begin{bmatrix}
      m_1 & 0   & \cdots & 0 \\
      0   & m_2 & \cdots & 0 \\
      \vdots & \vdots &\ddots & \vdots \\
      0 & 0 & \cdots & m_t
    \end{bmatrix}.
\end{align*}
This continuous representation restricts to the closed subgroup $\G_{K,\Delta}$, yielding a continuous representation $\G_{K,\Delta} \to \GL_t(\mathbb{Z}_p)$. Composing this with the homomorphism $\diag(\bm)\mapsto m_\alpha$ with $\alpha\in\Delta$, we get characters
\begin{align*}
    \chi_\alpha:\G_{K,\Delta}\to\mathbb{Z}^\times_p
\end{align*}

\begin{definition}
    We define the \emph{multivariate cyclotomic character} as $\chi_\Delta:=\det\circ\chi$.
\end{definition}

\begin{remark}
    For the group homomorphism $\eta_s:\G_{K_s}\to\G_{K,\Delta}$, the composition $\chi_\Delta\circ\eta_s$ is the usual $p$-adic cyclotomic character of $\G_{K_s}$.
\end{remark}

\begin{definition}
    Let $i\in\mathbb{Z}$; we define the \emph{Tate twist} $\mathbb{Z}_p(i)$ to be a free $\mathbb{Z}_p$-module
    \begin{align*}
      \mathbb{Z}_p(i)=\mathbb{Z}_pu^i
    \end{align*}
    with an action of $\G_{K,\Delta}$ by the formula
    \begin{align*}
        g\cdot u^i:=\chi^i_\Delta(g)u^i
    \end{align*}
    for all $g\in\G_{K,\Delta}$. For a $\mathbb{Z}_p$-module $M$, we define
    \begin{align*}
        M(i):=M\otimes_{\mathbb{Z}_p}\mathbb{Z}_p(i).
    \end{align*}
\end{definition}

Finally, we can define the Hodge-Tate period ring in the multivariate setup as follows: Let $\bi=(i_1,\cdots,i_t)\in\mathbb{N}^t$ and let $i=|\bi|=i_1+\cdots+i_t$. Then we have the following definition:

\begin{definition}\label{Pre-Hodge Tate Period Rings}
    The \emph{multivariate pre-Hodge-Tate period ring} $\mBHT$ is defined as
    \begin{align*}
        \pmBHT:=\bigoplus\limits_{i\in\mathbb{Z}}\C_\Delta(i)=\C_\Delta\left[t_\Delta,\frac{1}{t_\Delta}\right]
    \end{align*}
    The ring structure on $\mBHT$ is given by $cu^ic'u^j=cc'u^{i+j}$, and hence it is a graded ring.
\end{definition}

Note that $\mBHT$ is a subring $\C_\Delta(\!(t_\Delta)\!)$. Let us first recall some facts about the power series ring $\C_\Delta[\![t_\Delta]\!]$ over $\C_\Delta$ given in \cite{Powerseriesdim}. A ring $R$ is called \emph{strong finite type (SFT)} if for every ideal $\mathfrak{a}$ of $R[\![T]\!]$, there exists a finitely generated ideal $\mathfrak{b}\subset\mathfrak{a}$ and $k\in\mathbb{N}^*$ such that $a^k\in\mathfrak{b}$ for each $a\in\mathfrak{a}$. Note that $\C_\Delta$ is SFT if and only if $t=|\Delta|=1$. Then we have the following classical theorem of Arnold:

\begin{proposition}[\cite{Powerseriesdim},Theorem 1, Theorem 2]\label{Power Series dim}
    The power series ring $\C_\Delta[\![t_\Delta]\!]$ has the following characterisation of Krull dimension:
    \begin{align*}
        \dim\C_\Delta[\![t_\Delta]\!]=
        \begin{cases}
            1 & \text{ if } t=1\\
            \infty & \text{ if }t>1
        \end{cases}.
    \end{align*}
\end{proposition}

For every $s\in\mathfrak{S}$, we have a ring homomorphism
\begin{align*}
    \C_\Delta[\![t_\Delta]\!]&\to \C[\![t]\!]\\
     \sum\limits_{n=0}^\infty a_nt^n_\Delta&\mapsto \sum\limits_{n=0}^\infty \beta_s(a_n)t^n
\end{align*}
Its kernel is a prime ideal given by $\mathfrak{P}_s[\![t_\Delta]\!]=\left\{\sum\limits_{n=0}^\infty a_nt^n_\Delta\in\C_\Delta[\![t_\Delta]\!]\;\middle|\; a_n\in\mathfrak{P}_s\right\}$. The proof of \ref{Power Series dim} produces an infinite chain of prime ideal $\mathfrak{P}_s[\![t_\Delta]\!]^{(n)}$ such that
\begin{align*}
    \cdots\subset(\mathfrak{P}_s[\![t_\Delta]\!])^{(2)}\subset(\mathfrak{P}_s[\![t_\Delta]\!])^{(1)}\subset\mathfrak{P}_s[\![t_\Delta]\!]\subset \mathfrak{P}_s+t_\Delta\C_\Delta[\![t_\Delta]\!]
\end{align*}

Localising with respect to the multiplicative set $\{t_\Delta,t^2_\Delta,\cdots\}$, we obtain a chain of prime ideals 
\begin{align*}
    \cdots\subset (\mathfrak{P}_s(\!(t_\Delta)\!))^{(2)}\subset (\mathfrak{P}_s(\!(t_\Delta)\!))^{(1)}\subset \mathfrak{P}_s(\!(t_\Delta)\!)
\end{align*}

Therefore, we obtain

\begin{proposition}\label{Hodge-Tate Period rings are infinite dimensional}
    Either of the following two is true for the ring $\C_\Delta(\!(t_\Delta)\!)$:
    \begin{enumerate}
        \item[(i)] If $|\Delta|=1$, the the ring $\C_\Delta(\!(t_\Delta)\!)$ is a field;
        \item[(ii)] If $|\Delta|>1$, then the ring $\C_\Delta(\!(t_\Delta)\!)$ has infinite Krull dimension and has maximal ideal $\mathfrak{P}_s(\!(t_\Delta)\!),\;s\in\mathfrak{S}$ which is the kernel of $\C_\Delta(\!(t_\Delta)\!)\to \C(\!(t)\!)$.
    \end{enumerate}
\end{proposition}

\begin{proof}
    We only need to show $\C_\Delta$ is not SFT. For every $s\in\mathfrak{S}$, the ideal $\mathfrak{P}_s$ is not finitely generated. Suppose there exists a finitely generated ideal $\mathfrak{J}\subseteq\mathfrak{P}_s$ and a $k\in\mathbb{N}^*$ such that $a^k\in\mathfrak{J}$ for every $a\in\mathfrak{P}_s$. Since $\C_\Delta$ is a VNR ring, the ideal $\mathfrak{J}$ is generated by idempotent and is a direct summand. This means $\C_\Delta=e\C_\Delta\oplus(1-e)\C_\Delta$. Then we can write
    \begin{align*}
        a=ae+a(1-e)
    \end{align*}
    Note that $y=a(1-e)\in\mathfrak{P}_s$ and hence we have $y^k\in\mathfrak{J}=e\C_\Delta$. Then we get
    \begin{align*}
        ye=a(1-e)e=ae-ae=0\implies y^ke=0.
    \end{align*}
    But $y^k\in e$ and write $y^k=er$ for $r\in\C_\Delta$, we get
    \begin{align*}
        y^ke=0\implies e^2r=0\implies er=0\implies y^k=0\implies y=0
    \end{align*}
    since $\C_\Delta$ is reduced. Therefore, we have $a=ae+y=ae\in\mathfrak{J}$. Since $a$ is an arbitrary element of $\mathfrak{P}_s$, we get $\mathfrak{P}_s\subseteq\mathfrak{J}$ and hence $\mathfrak{P}_s=\mathfrak{J}$, a contradiction.
\end{proof}

Since $\pmBHT$ is a subring of $\C_\Delta(t_\Delta)$ and whose completion with respect to $t_\Delta$ is $\C_\Delta(\!(t_\Delta)\!)$, therefore we obtain

\begin{proposition}
    Either of the following two conditions is satisfied for the ring $\pmBHT$:
    \begin{enumerate}
        \item[(i)] If $|\Delta|=1$, then $\dim\pmBHT$ is 1.
        \item[(ii)] If $|\Delta|>1$, then $\dim\pmBHT=\infty$.
    \end{enumerate}
\end{proposition}

This shows that $\pmBHT$ is not well-behaved when $|\Delta|>1$; thus, we want to avoid this complication and we will refine the period ring $\pmBHT$ so that it is much more well-behaved and comparable to classical $p$-adic Hodge theory as follows:\\

For every $s\in\mathfrak{S}$, consider the map
\begin{align*}
    \Psi_s:\C_\Delta(\!(t_\Delta)\!)&\to \C(\!(t)\!),\\
    \sum\limits_{i=n}^\infty a_it^i_\Delta&\mapsto \sum\limits_{i=n}^\infty\beta_s(a_i)t^i.
\end{align*}
Suppose $x=\sum\limits_{i=n}^\infty a_it^i_\Delta\in\bigcap\limits_{s\in\mathfrak{S}}\ker\Psi_s$ is an arbitrary element; then for all $i\geq n$, we have $\beta_s(a_i)=0$ for all $s\in\mathfrak{S}$. Thus, $a_i=0$ for all $i\geq n$ and we get $x=0$. This shows that the collection $\{\ker\Psi_s\}_{s\in\mathfrak{S}}$ is a collection of maximal ideals such that $\bigcap\limits_{s\in\mathfrak{S}}\ker\Psi_s=\{0\}$, and this collection also exhausts the maximal ideals. Therefore, we obtain the profinite completion $\mCHT=\widehat{\C_\Delta(\!(t_\Delta)\!)}$ as given in Definition \ref{Profinite completion}. This ring is a VNR ring whose set of maximal ideals is in one-to-one correspondence with the embedding family $\mathfrak{S}$. Also note that this VNR ring is not isomorphic to $\widehat{\C}_\Delta(\!(t_\Delta)\!)$, the Laurent series over the profinite completion of $\C_\Delta$, as the latter is not a VNR ring by the same proof as in Proposition \ref{Hodge-Tate Period rings are infinite dimensional}.\\

\begin{definition}
    We define the \emph{multivariate Hodge-Tate period ring} $\mBHT$ to be the closure of $\pmBHT$ inside $\mCHT$ with respect to the profinite topology.
\end{definition}

It is easy to see that $\mCHT$ is the $t_\Delta$-adic completion of the total ring of quotients of $\mBHT$. We will now show that $\mBHT$ also has a similar decomposition into a direct sum as in Definition \ref{Pre-Hodge Tate Period Rings}.

\begin{proposition}
    We have $\mBHT=\bigoplus\limits_{i\in\mathbb{Z}}\widehat{\C}_\Delta(i)=\widehat{\C}_\Delta\left[t_\Delta,\frac{1}{t_\Delta}\right]$.
\end{proposition}

\begin{proof}
    We only need to show that the closure of $\C_\Delta\left[t_\Delta,\frac{1}{t_\Delta}\right]$ in $\mCHT$ with the profinite topology is equal to $\widehat{\C}_\Delta\left[t_\Delta,\frac{1}{t_\Delta}\right]$. Clearly, $\C_\Delta\left[t_\Delta,\frac{1}{t_\Delta}\right]\subset\widehat{\C}_\Delta\left[t_\Delta,\frac{1}{t_\Delta}\right]$, and $\widehat{\C}_\Delta\left[t_\Delta,\frac{1}{t_\Delta}\right]$ is a closed subset of $\mCHT$. Indeed, if $f(T)\in\mCHT$ belongs to the closure of $\widehat{\C}_\Delta\left[t_\Delta,\frac{1}{t_\Delta}\right]$, then for every finite subset $\mathfrak{S}'\subset\mathfrak{S}$, there exists a Laurent polynomial $f_{\mathfrak{S}'}(T)\in\widehat{\C}\left[t_\Delta,\frac{1}{t_\Delta}\right]$ such that $f(t_\Delta)-f_{\mathfrak{S}'}(t_\Delta)\in\bigcap\limits_{s\in\mathfrak{S}'}\ker\Psi_s$. Suppose $f(t_\Delta)=\sum\limits_{i=n}^\infty a_it^i_\Delta$; then there exists a finite subset $\mathfrak{S}'\subset \mathfrak{S}$ such that for all $s\in\mathfrak{S}$, we have $\beta_s(a_i)\neq 0$ for infinitely many $i\geq n$. But since $f_\mathfrak{S'}(t_\Delta)$ is a Laurent polynomial, suppose of degree $m$, then there exists $m'\geq m$ such that $\beta_s(a_{m'})\neq 0$. Therefore, $\Psi_s(f(t_\Delta)-f_{\mathfrak{S}'}(t_\Delta))\neq 0$ as the coefficient of $t^{m'}$ is non-zero. This gives a contradiction if all but finitely many $a_i$ are non-zero. Therefore, $\widehat{\C}_\Delta\left[t_\Delta,\frac{1}{t_\Delta}\right]$ is a closed subset of $\mCHT$.\\

    Let $f(t_\Delta)=\sum\limits_{i=n}^\infty a_it^i_\Delta\in\widehat{\C}_\Delta\left[t_\Delta,\frac{1}{t_\Delta}\right]$; then for each $i\geq n$ and each finite subset $\mathfrak{S}'\subset\mathfrak{S}$, there exists $a_{i,\mathfrak{S}'}\in\C_\Delta$ such that $a_i-a_{i,\mathfrak{S}'}\in\bigcap\limits_{s\in\mathfrak{S}'}\mathfrak{P}_s$. This means $f_{\mathfrak{S}'}(t_\Delta)=\sum\limits_{i=n}^\infty a_{i,\mathfrak{S}'}t^i_\Delta\in \C_\Delta\left[t_\Delta,\frac{1}{t_\Delta}\right]$ satisfies $f(t_\Delta)-f_{\mathfrak{S}'}(t_\Delta)\in\bigcap\limits_{s\in\mathfrak{S}'}\mathfrak{P}_s(\!(t_\Delta)\!)$. Therefore, the net $(f_{\mathfrak{S}'}(t_\Delta))_{\mathfrak{S'}\subset\mathfrak{S}}$ in $\C_\Delta\left[t_\Delta,\frac{1}{t_\Delta}\right]$ converges to $f(t_\Delta)$. This completes the proof.
\end{proof}

Now we will prove that the gradation of the filtered ring $\mBdr$ is the graded ring $\mBHT$.

\begin{proposition}\label{de Rham Hodge Tate Proposition}
    We have $\gr\mBdr=\mBHT$.
\end{proposition}

\begin{proof}
    Written multiplicatively, we have $\mathbb{Z}_p(i)=\mathbb{Z}_p t^i_\Delta$ by replacing $u=t_\Delta$. Also, for $\lambda\in\mathbb{Z}_p$, we have
    \begin{align*}
        \log([\varepsilon]^\lambda)=\lambda\log([\varepsilon])=\lambda t_\Delta.
    \end{align*}
    For the cyclotomic character $\chi_\Delta$, we have $g\cdot t^i_\Delta=\chi^i_\Delta(g)t^i_\Delta$, and hence for all $i\in\mathbb{Z}$,
    \begin{align*}
        \Fil^i_\Delta\mBdr=t^i_\Delta\mBdrp=\mBdrp(i).
    \end{align*}
    Therefore,
    \begin{align*}
        \gr^i\mBdr&=\bigoplus\limits_{i\in\mathbb{Z}}\Fil^i_\Delta\mBdr/\Fil^{i+1}_\Delta\mBdr\\
                &=\bigoplus\limits_{i\in\mathbb{Z}}\mBdrp(i)/\mBdrp(i+1)\\
                &=\bigoplus\limits_{i\in\mathbb{Z}}\widehat{\C}_\Delta(i)=\mBHT.
    \end{align*}
\end{proof}

\begin{corollary}
    We have $\gr\pmBdr=\pmBHT$.
\end{corollary}

\begin{proof}
    Consider the filtration on the multivariate pre-de Rham period ring $\pmBdr$ induced by the subspace topology, which is defined by setting $\Fil^i_\Delta\pmBdr = \Fil^i_\Delta\mBdr \cap \pmBdr$ for each $i \in \mathbb{Z}$. By definition, the associated grading construction is compatible with this filtration. Since the graded pieces of the completed period ring satisfy $\gr^i\mBdr = \widehat{\C}_\Delta(i)$ as shown in Proposition \ref{de Rham Hodge Tate Proposition}, intersecting the filtered components with the pre-integral substructures yields exactly the uncompleted terms. Specifically, we have $\Fil^i_\Delta\pmBdr / \Fil^{i+1}_\Delta\pmBdr \simeq \C_\Delta(i)$ for each $i \in \mathbb{Z}$. Summing over all grading pieces, we obtain $\gr\pmBdr = \bigoplus_{i\in\mathbb{Z}}\C_\Delta(i) = \pmBHT$, which completes the proof.
\end{proof}


\section{Admissible Representations}\label{Admissible Representations}

 The Definition \ref{Free Representations definition} makes sense when we replace $\G_K$ by $\G_{K,\Delta}$ and with arbitrary $B$. Then we have the following proposition about classifying free representations using non-abelian cohomology:

\begin{proposition}[\cite{Fontaine}, Proposition 3.7]
    Suppose $d$ is a positive integer. There is a one-to-one correspondence 
    \begin{align*}
        \text{(Free $B$-representations of $G_{K,\Delta}$)}/\simeq\leftrightarrow \cH^1_{\text{cont}}(\G_{K,\Delta},\GL_d(B)).
    \end{align*}
    In this correspondence, trivial representations go to the distinguished point of $\cH^1_{\text{cont}}(\G_{K,\Delta},\GL_d(B))$.
\end{proposition}

Let $B$ be a topological ring on which $G_{K,\Delta}$ acts continuously, and $F$ is a closed subring of $B^{G_{K,\Delta}}$ that is a field. Suppose $C$ is the total ring of fractions of $B$, which we will assume to be a VNR ring, then the action of $\G_{K,\Delta}$ extends naturally to $C$. Suppose the following condition is satisfied:\\

\emph{For every prime ideal $\mathfrak{Q}\in\spec C$, there exists $s(\mathfrak{Q})\in\mathfrak{S}$ such that $\G_{K_{s(\mathfrak{Q})}}$ acts on $C_{\mathfrak{Q}}$ which is compatible with the maps $\eta_s$ and $C\to C_{\mathfrak{Q}}$}.\\

Then we have the following definition:

\begin{definition}
    A reduced ring $B$ is called \emph{$(F,\G_{K,\Delta})_\Delta$-regular} if the following conditions hold:
    \begin{enumerate}
        \item[(i)] $C$ is VNR ring,
        \item[(ii)] $E=B^{G_{K,\Delta}}=C^{\G_{K,\Delta}}$ and $C$ is faithfully flat over $E$.
        \item[(iii)] For $b\in B$, a non-zero divisor, if for any $g\in\G_{K,\Delta}$, there exists $\lambda=\lambda(g)\in F$ such that $g(b)=\lambda b$, then $b$ is invertible in $B$.
    \end{enumerate}
\end{definition}

\begin{remark}
    The inclusion $E\subset B\subset C$ induces a surjective map $\spec C\to \spec B\to \spec E$. If $C$ is faithfully flat over $E$, then $\spec B\to\spec E$ is also surjective and hence $B$ is also faithfully flat over $E$. Using Lemma \ref{fixedpointofVNR}, we see that $E$ is also a VNR ring.
\end{remark}

\begin{definition}
    A representation $V$ of $\G_{K,\Delta}$ is called \emph{$B$-admissible} if $B\otimes_F V$ is a trivial representation of $\G_{K,\Delta}$.
\end{definition}

Recall that a $B$-representation is called trivial if it contains a basis fixed by $\G_{K,\Delta}$. Let $V$ be any $F$-representation of $\G_{K,\Delta}$, then $B\otimes_F V$ is a free $B$-representation of $\G_{K,\Delta}$. The action of $\G_{K,\Delta}$ is given by $g(\lambda\otimes x)=g(\lambda)\otimes g(x)$. Define
\begin{align*}
    \bD_B(V):=(B\otimes_F V)^{\G_{K,\Delta}}
\end{align*}

Then we hae a natural $\G_{K,\Delta}$-equivariant linear map
\begin{align*}
    \alpha_V: B\otimes_E\bD(V)&\to B\otimes_F V\\
            \lambda\otimes x  &\mapsto \lambda x
\end{align*}
where the action on the left-hand side is given by $g(\lambda\otimes x)=g(\lambda)\otimes x$.\\

To proceed to find Tannakian subcategories, we need standard results on faithfully flat descent and on projective modules over VNR rings.

\begin{proposition}[\cite{Wed-Gor},Proposition 14.48;\cite{stacks-project},Theorem 10.95.6]\label{Faithfully Flat Descent of Modules} Let $A\to A'$ be a faithfully flat ring homomorphism, let $M$ be an $A$-module and write $M':= M \otimes_A A'$. Then
\begin{enumerate}
    \item[(i)] If $M'$ is flat over $A'$ , then $M$ is flat over $A$.
    \item[(ii)] If $M'$ is of finite type over $A'$ , then $M$ is of finite type over $A$.
    \item[(iii)] If $M'$ is of finite presentation over $A'$ , then $M$ is of finite presentation over $A$.
    \item[(iv)] If $M'$ is locally free of rank n over $A'$, then $M$ is locally free of rank $n$ over $A$.
    \item[(v)] If $M'$ is projective over $A'$, then $M$ is projective over $A$.
\end{enumerate}
\end{proposition}

\begin{proposition}[\cite{Good},Proposition 3.10]\label{Projective module over a VNR ring}
    Let $M$ and $N$ be finitely generated projective modules over a VNR $A$, then:
    \begin{enumerate}
        \item[(i)] $M$ is a submodule of $N$ if and only if $\dim_{A/\mathfrak{p}}(M/{\mathfrak{p}}M)\leq \dim_{A/\mathfrak{p}}(N/\mathfrak{p}N)$ for all $\mathfrak{p}\in\spec A$.
        \item[(ii)] $M$ is isomorphic to $N$ if and only if $\dim_{A/\mathfrak{p}}(M/{\mathfrak{p}}M)= \dim_{A/\mathfrak{p}}(N/\mathfrak{p}N)$ for all $\mathfrak{p}\in\spec A$.
    \end{enumerate}
\end{proposition}

\begin{lemma}
    Let $R$ be a reduced ring and $\mathfrak{P}\in\spec R$ be a minimal prime ideal, then $R_{\mathfrak{P}}$ and $\Frac(R/\mathfrak{P})$ are isomorphic. In particular for VNR ring $R_{\mathfrak{P}}$ and $\Frac(R/\mathfrak{P})$ are isomorphic.
\end{lemma}

\begin{proof}
    Let $\eta_{\mathfrak{P}}:R\to R/\mathfrak{P}$ be the canonical mapping. Then define
    \begin{align*}
        \phi_{\mathfrak{P}}:R_{\mathfrak{P}}&\to \Frac(R/\mathfrak{P})\\
                                 \frac{x}{s}&\mapsto \eta_{\mathfrak{P}}(s)^{-1}\eta_{\mathfrak{P}}(x)
    \end{align*}
    This ring homomorphism makes sense as if $s\notin R\setminus\mathfrak{P}$, then $\eta_{\mathfrak{P}}(s)\neq 0$. But $\mathfrak{P}$ is a minimal prime ideal, so this element is invertible. $\phi_{\mathfrak{P}}$ is clearly surjective. Also $R_{\mathfrak{P}}$ is a field, the map $\phi_{\mathfrak{P}}$ is an isomorphism.
\end{proof}

\begin{remark}
    Since exactness is a local property, we have for a VNR ring $R$, if 
    \begin{align*}
        0\to R^{d'}\to R^d\to R^{d''}\to 0
    \end{align*}
    is a short exact sequence if and only if $d'+d''=d$.
\end{remark}

We also need some linear algebra over commutative rings, in particular VNR rings. This has been studied extensively in \cite{William}. Let us recall some facts that are needed for our purpose:\\

Let $R$ be a commutative ring with unity and let $A=[a_{i,j}]\in M_{m\times n}(R)$. Then for each $t\in\mathbb{Z}$ we define the ideals $I_t(A)$ as follows: Set $r=\min\{m,n\}$, then 
\begin{itemize}
    \item $I_t(A)=0$ for each $t>r$,
    \item $I_t(A)=R$ for each $t\leq 0$,
    \item $I_t(A)$ is the ideal generated by determinant of $t\times t$ sub-matrix of $A$ for each $t$ satisfying $1\leq t\leq r$.
\end{itemize}
Then we have an increasing chain of ideals of $R$ as
\begin{align*}
    (0)=I_{r+1}\subseteq I_r(A)\subseteq \cdots\subseteq I_1(A)\subseteq I_0(A)=R.
\end{align*}
Computing the annihilator of the above chain, we get an increasing chain of ideals
\begin{align*}
    (0)=\Ann_R(R)\subseteq\Ann_R(I_1(A))\subseteq\Ann_R(I_2(A))\subseteq\cdots\subseteq\Ann_R(I_r(A))\subseteq\Ann_R((0))=R.
\end{align*}
Then we have the following definition:

\begin{definition}[\cite{William},Definition 4.10]
    Let $A\in M_{m\times n}(R)$. The \emph{rank} of $A$ is defined as
    \begin{align*}
        \rank(A):=\max\{t\mid \Ann_R(I_t(A))=0\}.
    \end{align*}
\end{definition}

Then we have a classical theorem due to N. McCoy regarding a system of linear equations
\begin{theorem}[\cite{William}, Theorem 5.3]\label{solution of homogeneous systems}
    Let $A\in M_{m\times n}(R)$. The homogeneous system of linear equations $AX=0$ has a non-trivial solution if and only if $\rank(A)<n$.
\end{theorem}

\begin{lemma}\label{fixed point lemma}
    If $R$ is a VNR ring, $G$ is a group acting $R$-linearly on a module $M$, and $N$ is a finitely presented $R$-module. Then $G$ acts on $N\otimes_R M$ via the second factor, then we have
    \begin{align*}
        N\otimes_R M^G&\to (N\otimes_R M)^G\\
          b\otimes a  &\mapsto b\otimes a
    \end{align*}
    is an isomorphism. 
\end{lemma}

\begin{proof}
    Consider the map
    \begin{align*}
       u: M&\to \prod\limits_{g\in G}M\\
        a&\mapsto (a-g(a))_{g\in G}
    \end{align*}
    It's kernel if $M^G$ and hence consider the exact sequence
    \begin{align*}
        0\to M^G\to M\to\im u\to 0 
    \end{align*}
    where $\im(u)=\left\{(b_g)\in\prod\limits_{g\in G}M\;\middle|\; \text{ there exists }a\in A\text{ such that }b_g=a-g(a)\text{ for all }g\in G\right\}$. Taking tensor product with $N$, we get
    \begin{align*}
        0\to N\otimes_R M^G\to N\otimes_R M\to N\otimes_R\im u\to 0
    \end{align*}
    since every module over a VNR ring is flat. But since $N$ is finitely presented, we have $N\otimes_R\prod\limits_{g\in G}M\cong \prod\limits_{g\in G}(N\otimes_R M)$, we get the kernel is exactly the kernel of
    \begin{align*}
        \id_N\otimes u: N\otimes_R M\to \prod\limits_{g\in G}(N\otimes_R M)
    \end{align*}
    which is exactly $(N\otimes_R M)^G$ and therefore the proof.
\end{proof}

We can now turn to our main theorem in this section:

\begin{theorem}
    Assume that $B$ is $(F,G_{K,\Delta})_\Delta$-regular. Then:
    \begin{enumerate}
        \item[(i)] For any $F$-representation $V$ of $\G_{K,\Delta}$, the map $\alpha_V$ is an isomorphism if and only if $\bD_B(V)$ is a free $E$-module of rank $\dim V$, which is equivalent to $V$ being $B$-admissible.
        \item[(ii)] The full subcategory $\Rep^B_F(\G_{K,\Delta})$ is a Tannakian subcategory of $\Rep_k(\G_{K,\Delta})$.
    \end{enumerate}
\end{theorem}

\begin{proof}
    Suppose $\alpha_V: B\otimes_E\bD_B(V)\to B\otimes_F V$ is an isomorphism. Since $V\otimes_FB$ is a free module, it is locally free of constant rank and projective. Using Proposition \ref{Faithfully Flat Descent of Modules}, we see that $\bD_B(V)$ is also locally free of constant rank $\dim V$ and projective. Since $E$ is a VNR ring, Proposition \ref{Projective module over a VNR ring} implies that $\bD_B(V)$ is a free $E$-module of rank $\dim V$. Furthermore, by definition, $B\otimes_E\bD_B(V)$ is a trivial $B$-representation; hence, $V$ is $B$-admissible. \\

    Conversely, suppose $\bD_B(V)$ is a free $E$-module. We claim that $\alpha_V:B\otimes_E\bD_B(V)\to B\otimes_F V$ is injective. From the commutative diagram
    \[
    \begin{tikzcd}
        B\otimes_E\bD_B(V)\arrow[r,"\alpha_{V,B}"]\arrow[d]& B\otimes_F V\arrow[dd]\\
        B\otimes_E\bD_C(V)                      \arrow[d]  &             \\
        C\otimes_E\bD_C(V)\arrow[r,"\alpha_{V,C}"]& C\otimes_F V
    \end{tikzcd}
    \]
    Since the vertical arrows are injective and $C$ is the total ring of fractions of $B$, it is enough to prove that $\alpha_{V,C}$ is injective. Therefore, it suffices to assume that $B$ is a total ring of fractions. Because $B$ is faithfully flat over $E$, the induced map $\spec B\to \spec E$ is surjective; thus, every $\mathfrak{P}\in\spec E$ can be written as $\mathfrak{Q}\cap E$ for some prime ideal $\mathfrak{Q}\in\spec B$, use $E$ is a VNR ring. We have the injective map $E_{\mathfrak{P}}\to B_{\mathfrak{Q}}$, and $F$ is a closed subfield of $E_{\mathfrak{P}}$. Since $B$ is a VNR ring, the ring $B_{\mathfrak{Q}}$ is a field and hence is $(F,G_{K_{s(\mathfrak{Q})}})$-regular in sense of \cite{Fontaine}, Definition 3.9, and hence the map
    \begin{align*}
        \alpha_{V,\mathfrak{Q}}: B_{\mathfrak{Q}}\otimes_{E_{\mathfrak{P}}}\bD_{B_{\mathfrak{Q}}}(V)\to B_{\mathfrak{Q}}\otimes_F\bD_{B_{\mathfrak{Q}}}(V)
    \end{align*}
    is injective, where $\bD_{B_{\mathfrak{Q}}}(V):=(B_{\mathfrak{Q}}\otimes_F V)^{G_{K_{s(\mathfrak{Q})}}}$. Thus, it is enough to prove that the maps
    \begin{align*}
        B\otimes_E\bD_B(V)\to\prod\limits_{\mathfrak{Q}\in\spec B}(B_{\mathfrak{Q}}\otimes_{E_{\mathfrak{Q}\cap E}}\bD_{B_{\mathfrak{Q}}}(V))
    \end{align*}
    and 
    \begin{align*}
        B\otimes_F V\to \prod\limits_{\mathfrak{Q}\in\spec B}(B_{\mathfrak{P}}\otimes_F V)
    \end{align*}
    are injective. Note that $B\to\prod\limits_{\mathfrak{Q}\in\spec(B)}(B/\mathfrak{Q})\to\prod\limits_{\mathfrak{Q}\in\spec(B)}B_{\mathfrak{Q}}$ is injective. As $V$ is a finitely presented $F$-module, we have $\prod\limits_{\mathfrak{Q}\in\spec B}B_{\mathfrak{Q}}\otimes_F V=\left(\prod\limits_{\mathfrak{Q}\in\spec B}B_{\mathfrak{Q}}\right)\otimes_F V$, which shows the second map is injective.\\

    Taking $\G_{K,\Delta}$-fixed point of the second map, we get an injective map
    \begin{align*}
        \bD_B(V)\to \left(\prod\limits_{\mathfrak{Q}\in\spec B}(B_{\mathfrak{Q}}\otimes_F V)\right)^{\G_{K,\Delta}}
    \end{align*}
   But we have $\left(\prod\limits_{\mathfrak{Q}\in\spec B}B_{\mathfrak{Q}}\otimes_F V\right)^{\G_{K,\Delta}}\to \prod\limits_{\mathfrak{Q}\in\spec B}(B_{\mathfrak{Q}}\otimes_F V)^{\G_{K_{s(\mathfrak{Q})}}}$ is just the identity map, because the action is compatible with $\eta_{s(\mathfrak{Q})}:\G_{K_s}\to\G_{K,\Delta}$, the map $\bD_B(V)\to\prod\limits_{\mathfrak{Q}\in\spec{B}}\bD_{B_{\mathfrak{Q}}}(V)$ is an injective. Taking the tensor product functor $B\otimes_E-$, we get an injective map
   \begin{align*}
       B\otimes_E\bD_B(V)\to B\otimes_E\prod\limits_{\mathfrak{Q}\in\spec B}\bD_{B_{\mathfrak{Q}}}(V)\to \prod\limits_{\mathfrak{Q}\in\spec B}(B_{\mathfrak{Q}}\otimes_{E_{\mathfrak{Q}\cap E}}\bD_{B_{\mathfrak{Q}}}(V))
   \end{align*}
    because $E$ is a VNR ring and hence every module over $E$ is flat. Therefore, the map $B\otimes_E\bD_B(V)\to \prod\limits_{\mathfrak{Q}\in\spec B}\left(B_{\mathfrak{Q}}\otimes_{E_{\mathfrak{Q}\cap E}}\bD_{B_{\mathfrak{Q}}}(V)\right)$ is injective which is what we want to prove. If $\bD_B(V)$ is free, then from Theorem \ref{solution of homogeneous systems}, it follows that the determinant $b$ of $\alpha_V$ is a non-zero divisor.\\

    We will now prove that $\alpha_V$ is an isomorphism if $\rank_E\bD_B(V)=\dim_F V$, which is equivalent to $b$ being a unit in $B$. Let $\{v_1,\dots,v_d\}$ be a basis of $V$ and $\{e_1,\dots,e_d\}$ be a basis of $\bD_B(V)$. Then we have
    \begin{align*}
        e_j=\sum\limits_{i=1}^db_{ij}v_i
    \end{align*}
    such that $b=\det[b_{i,j}]$. Let $v=v_1\wedge\dots\wedge v_d\in \bigwedge^d_FV=Fv$. Then $g(v)=\eta(g)v$, where $\eta:G_{K,\Delta}\to F^\times$ is a homomorphism. Similarly, let $e=e_1\wedge\dots\wedge e_d\in\wedge^d_E\bD_B(V)$. We then have $e=bv$, and thus $e=g(e)=g(b)g(v)=g(b)\eta(g)v$. Consequently, $g(b)=\eta(g)^{-1}b$ for all $g\in G_{K,\Delta}$. By the $(F,G_{K,\Delta})_\Delta$-regularity of $B$, the element $b\in B$ is a unit.\\

   The rest of the proof is very similar to that of Theorem 3.14 of \cite{Fontaine}. Let us recall the proof for the sake of completeness. Suppose we have an exact sequence
   \begin{align}\label{exactnessI}
       0\to V'\to V\to V''\to 0
   \end{align}
   be an exact sequence of $F$-vector spaces, then we have
   \begin{align*}
       0\to B\otimes_F V'\to B\otimes_F V\to B\otimes_F V''
   \end{align*}
   an exact sequence of $B$-modules. Taking $\G_{K,\Delta}$-fixed points we get an exact sequence
   \begin{align}\label{exactnessII}
       0\to \bD_B(V')\to \bD_B(V)\to \bD_B(V'')
   \end{align}
   of $E$-modules. These modules may not be free, but if all of $\bD_B(V'),\bD_B(V)$ and $\bD_B(V'')$ are free and $V$ is admissible, we have by the first part $\rank_E\bD_B(V')\leq \dim_FV, \rank_E\bD_B(V'')\leq\dim_FV''$ and $\rank_E\bD_B(V)=\dim_F V$. By (\ref{exactnessI}), we get 
   \begin{align*}
       \rank_E\bD_B(V)=\dim_FV=\dim_FV'+\dim_FV''\geq \rank_E\bD_B(V')+\rank_E\bD_B(V'')\geq \rank_E\bD_B(V).
   \end{align*}
   Therefore, we have $\rank_E\bD_B(V')=\dim_FV'$ and $\rank_E\bD_B(V'')=\dim_F(V'')$. But $E$ is a VNR ring, therefore the sequence (\ref{exactnessII}) is also exact. In particular, $\bD_B(-)$ is an exact functor when restricted to $\Rep^B_F(\G_{K,\Delta})$.\\

       We now prove that $\Rep^B_F(\G_{K,\Delta})$ is closed under tensor products. Let $V_1$ and $V_2$ be two $B$-admissible representations of $\G_{K,\Delta}$. Then we have the isomorphisms:
    \begin{align*}
        B\otimes_E\bD_B(V_1) &\cong B\otimes_F V_1 \\
        B\otimes_E\bD_B(V_2) &\cong B\otimes_F V_2
    \end{align*}
    It follows that:
    \begin{align*}
        \bD_B(V_1\otimes_F V_2) &= (B\otimes_F (V_1\otimes_F V_2))^{\G_{K,\Delta}} 
                               = ((B\otimes_F V_1)\otimes_B (B\otimes_F V_2))^{\G_{K,\Delta}} \\
                               &= ((B\otimes_E\bD_B(V_1))\otimes_B (B\otimes_E\bD_B(V_2)))^{\G_{K,\Delta}}
                               = (B\otimes_E\bD_B(V_1)\otimes_E\bD_B(V_2))^{\G_{K,\Delta}}\\
                               &=B^{\G_{K,\Delta}}\otimes_E\bD_B(V_1)\otimes_E\bD_B(V_2)
                               =\bD_B(V_1)\otimes_E\bD_B(V_2)
    \end{align*}
    Here we have used that $\bD_B(V_1)\otimes_E\bD_B(V_2)$ is free and hence finitely presented, and that $E$ is a VNR ring in Lemma \ref{fixed point lemma}. Since $\bD_B(V_1)$ and $\bD_B(V_2)$ are free $E$-modules, their tensor product $\bD_B(V_1\otimes_F V_2)$ is also a free $E$-module. Moreover, because $\rank_E\bD_B(V_1\otimes_F V_2) = \dim_F(V_1)\dim_F(V_2) = \dim_F(V_1\otimes_F V_2)$, the tensor product of two $B$-admissible representations is $B$-admissible. The unit and duals of $B$-admissible representations can be shown to be $B$-admissible exactly as in \cite{Fontaine}, which shows that $\Rep^B_F(\G_{K,\Delta})$ forms a Tannakian subcategory of $\Rep_F(\G_{K,\Delta})$.
\end{proof}

\begin{remark}
    This theorem also answers affirmatively the question posed at the end of Section 4 about the freeness of $\mbDdR(V)$ in \cite{BCM}, with the modified de Rham period rings. We also showed that $\alpha_V$ is always injective, regardless of whether $\bD_B(V)$ is free. However, the important point to note is that the period ring defined in \cite{BCM} does not satisfy all the conditions of $(F,\G_{K,\Delta})_{\Delta}$-regularity, whereas we will see that our reformulation satisfies these conditions.
\end{remark}

\begin{conjecture}\label{conj}
    $\bD_B(V)$ is finitely presented.
\end{conjecture}

If \ref{conj} is true, then $\bD_B(V)$ is a projective $E$-module and hence $\bD_B(V)$ is free if and only if $\dim_{\mathfrak{Q}}\bD_{B_\mathfrak{Q}}(V)$ is constant for all $\mathfrak{Q}\in\spmin B$.


\section{Hodge-Tate and de Rham Representations}

We start this section by proving the classical Ax-Sen-Tate theorem in multivariate setups. For that, let us prove a lemma that is needed for the proof.

\begin{lemma}\label{Fixed ring lemma}
    Let $K_1,\cdots,K_t\subset \C_\Delta$ are fields over a p-adic field $K_0$ which may not be finite over $\mathbb{Q}_p$ and let $\G_{K,\Delta}=\G_{K_1}\times\cdots\times\G_{K_t}$ be the direct product of the absolute Galois groups. Then the completion of $(\C\otimes_{K_0}\cdots\otimes_{K_0}\C)^{\G_{K,\Delta}}$ and $\C^{\G_{K,\Delta}}_\Delta$ are isomorphic with respect to the semi-norm $|\bullet|=\sup\limits_{s\in\mathfrak{S}}|\beta_s(\bullet)|$.
\end{lemma}

\begin{proof}
    Note that $\C^{\G_{K,\Delta}}_\Delta$ is complete with respect to $|\bullet|$. Indeed, if $(x_n)_{n\in\mathbb{N}}$ is a Cauchy sequence in $\C^{G_{K,\Delta}}_\Delta$ which converges to $x\in\C_\Delta$, then for every $\epsilon>0$, there exists $N\in\mathbb{N}$ such that for all $g\in\G_{K,\Delta}$ we have $|x_n-g(x)|=|g(x_n)-g(x)|=|x_n-x|<\epsilon$. This is true for all $\epsilon>0$, we have $g(x)=x$ and hence $x\in\C^{\G_{K,\Delta}}_\Delta$. Also, $\C^{\G_{K,\Delta}}_\Delta\cap(\C\otimes_{K_0}\cdots\otimes_{K_0}\C)=(\C\otimes_{K_0}\cdots\otimes_{K_0}\C)^{\G_{K,\Delta}}$ is dense in $\C^{\G_{K,\Delta}}_\Delta$. Thus its completion must be equal to $\C^{\G_{K,\Delta}}_\Delta$.
\end{proof}

As a consequence, we obtain the Ax-Sen-Tate Theorem (Proposition 4.46 of \cite{Fontaine}) in the multivariate case

\begin{proposition}\label{Multivariate Tate-Sen Theorem}
    Set up as in Lemma \ref{Fixed ring lemma}, we have
    \begin{align*}
        \C_{\Delta}(i)^{\G_{K,\Delta}}=
        \begin{cases}
            K_\Delta & \text{ if } i=0\\
            0 & \text{ if } i\neq 0
        \end{cases}
    \end{align*}
    where $K_\Delta=K_1\widehat{\otimes}_{K_0}\cdots\widehat{\otimes}_{K_0}K_t$ is the completed tensor product, which equals the completion of the usual tensor product with respect to the seminorm $|\bullet|$.
\end{proposition}

\begin{proof}
   First, suppose that $i=0$, then we will first show that $(\C\otimes_{K_0}\cdots{\otimes}_{K_0}\C)^{\G_{K,\Delta}}=K_1\otimes_{K_0}\cdots\otimes_{K_0}K_t$. First, we will show for $t=2$. From Lemma \ref{fixed point lemma} and \cite{Fontaine} Proposition 4.9 we get
   \begin{align*}
       (\C\otimes_{K_0}\C)^{\G_{K_1}\times\G_{K_2}}&=\left((\C\otimes_{K_0}\C)^{\G_{K_1}}\right)^{\G_{K_2}}\\
                        &=(\C^{\G_{K_1}}\otimes_{K}\C)^{\G_{K_2}}\\
                        &=\C^{K_1}\otimes_{K_0}\C^{K_2}=\hat{K}_1\otimes_{K_0}\hat{K}_2
   \end{align*}
   From induction, we get
   \begin{align*}
       (\C\otimes_{K_0}\cdots\otimes_{K_0}\C)^{\G_{K,\Delta}}=\widehat{K}_1\otimes_{K_0}\cdots\otimes_{K_0}\hat{K}_t
   \end{align*}

    From Lemma \ref{Fixed ring lemma} we get $\C^{\G_{K,\Delta}}_{\Delta}=K_\Delta$.
    For $i\neq 0$, note that its localisation at a prime ideal corresponding to $s\in\mathfrak{S}$ is $\C(i)^{\G_{K_s}}$, which is zero by the classical Ax-Sen-Tate Theorem. This is because the image of $\boldsymbol{\zeta}_{\bp^\bn}$ is a system of compatible roots of unity under the map $\beta_s$, since $\C_\Delta(i)=\C_\Delta\otimes_{\mathbb{Z}_p}\mathbb{Z}_p(i)$, where $\mathbb{Z}_p(i)$ is the tate twist for $\G_{K,\Delta}$ which maps to the tate twist $\mathbb{Z}_p(i)$ for $\G_{K_s}$ for every $s\in\mathfrak{S}$. Therefore, $\C_\Delta(i)^{\G_{K,\Delta}}=0$ for $i\neq 0$.
\end{proof}

\begin{corollary}\label{Corollary of Multivariate Tate-Sen Theorem}
    With the above setup, we have
    \begin{align*}
       \widehat{\C}_\Delta(i)^{\G_{K,\Delta}}&=
                   \begin{cases}
                     \text{ closure of $K_\Delta$ inside $\widehat{\C}_\Delta$} &\text{ if } i=0\\
                     0 &\text{ if } i\neq 0
                   \end{cases}.
    \end{align*}
\end{corollary}

\begin{proof}
    The result is easy because $\C_\Delta$ is dense in $\widehat{\C}_\Delta$. Note that if $K_1,\cdots,K_t$ are finite extensions of $K_0$, then the closure of $K_\Delta$ is equal to $K_\Delta$.
\end{proof}

Now we will prove the multivariate period rings $\mBHT$ and $\mBdr$ are $(\mathbb{Q}_p,\G_{K,\Delta})$-regular. From now on, we will assume $K_1,\cdots,K_t$ are finite extensions of $K_0$.

\begin{proposition}
    The rings $\mBHT$ and $\mBdr$ are $(\mathbb{Q}_p,\G_{K,\Delta})_\Delta$-regular.
\end{proposition}

\begin{proof}
    Let us first show that $\mBHT$ is $(\mathbb{Q}_p,\G_{K,\Delta})_\Delta$-regular. The total ring of quotients of $\mBHT$ has $t_\Delta$-adic completion $\mCHT$, which is a VNR ring. Since $\mCHT$ is the profinite completion of $\C_\Delta(\!(t_\Delta)\!)$, it is enough to show that $\C_\Delta^{\G_{K,\Delta}}(\!(t_\Delta)\!)=K_\Delta$. Suppose $b=\sum\limits_{i\in\mathbb{Z}}c_it^i$, where $c_i\in\C_\Delta$, is an arbitrary element in $\C^{\G_{K,\Delta}}_\Delta((t))$. Then for $g\in\G_{K,\Delta}$, we have
    \begin{align*}
        g(b)=\sum\limits_{i\in\mathbb{Z}}g(c_i)\chi_\Delta^i(g)t^i=\sum\limits_{i\in\mathbb{Z}}c_it^i.
    \end{align*}
    Therefore, for all $i\in\mathbb{Z}$, we have $c_it^i\in\widehat{\C}_\Delta(i)^{\G_{K,\Delta}}$, and using Corollary \ref{Corollary of Multivariate Tate-Sen Theorem}, we obtain $\C^{\G_{K,\Delta}}_\Delta(\!(t_\Delta)\!)=K_\Delta$ (We will prove the result only for this type of element; for others, we must take the limits). It is obvious that $\mCHT$ is faithfully flat over $K_\Delta$. Finally, assume $b=\sum\limits_{i\in\mathbb{Z}}c_it^i\in\mBHT$ is not a zero divisor such that $g(b)=\alpha b$ for some $\alpha\in\mathbb{Q}_p$. Then we get $g(b)=\sum\limits_{i=-\infty}^\infty g(b_i)\chi^i_\Delta(g)t^i=\sum\limits_{i=-\infty}^\infty \alpha b_it^i$. Then in the total ring of fractions of $\mBHT$, we have $g\left(\frac{b_i}{b_j}\right)$ and thus $g\left(\frac{b_j}{b_i}\right)=\chi_\Delta^{i-j}\frac{b_j}{b_i}$. Therefore, $\frac{b_j}{b_i}\in\left(\widehat{\C}_\Delta(j-i)\right)^{\G_{K,\Delta}}=\begin{cases}
        0 &\text{ if }i\neq j,\\
        K_\Delta &\text{ if }i=j
    \end{cases}$. This shows that $b=b_{i}t^i$, where $b_i\in\widehat{\C}_\Delta$ is a non-zero divisor.\\

    Next, we want to prove that $\mBdr$ is $(\mathbb{Q}_p,\G_{K,\Delta})$-admissible. We have already seen that $\mBdr$ is a VNR ring. Thus, we next want to show that $\mBdr^{\G_{K,\Delta}}=K_\Delta$. We have the following chain of subrings: $K_\Delta\subset \overline{K}_\Delta\subset\mBdrp\subset\mBdr$. Let $b\neq 0$ be an element of $\mBdr$; then there exists $i\in\mathbb{Z}$ such that $b\in\Fil^i_\Delta\mBdr$ but $b\notin\Fil^{i+1}_\Delta\mBdr$. Let $\overline{b}\in\gr^i\mBdr=\widehat{\C}_\Delta(i)$ be its image; then $\overline{b}\in\widehat{\C}_\Delta(i)^{\G_{K,\Delta}}=\begin{cases}
        0 &\text{ if }i\neq 0,\\
        K_\Delta &\text{ if }i=0
    \end{cases}$ is non-zero. Therefore, $i=0$ and hence $\overline{b}\in K_\Delta\subset\mBdrp$. Now $b-\overline{b}\in\mBdr^{\G_{K,\Delta}}$ and $b-\overline{b}\in(\Fil^i_\Delta\mBdr)^{\G_{K,\Delta}}$ for some $i\geq 1$. This shows that $b-\overline{b}=0$, proving that $b=\overline{b}\in K_\Delta$. We also have that $\spec\mBdr\to \spec K_\Delta$ is surjective; thus, $\mBdr$ is faithfully flat over $K_\Delta$.
\end{proof}

\begin{definition}
   \begin{enumerate}
       \item[(i)] A $p$-adic representation $V$ of $\G_{K,\Delta}$ is called \emph{Hodge-Tate} if it is $\mBHT$ admissible.
       \item[(ii)] A $p$-adic representation $V$ of $\G_{K,\Delta}$ is called \emph{de Rham} if it is $\mBdr$ admissible.
   \end{enumerate}
\end{definition}

For any $p$-adic representation, define 
\begin{align*}
    \mbDHT(V):=(\mBHT\otimes_{\mathbb{Q}_p}V)^{\G_{K,\Delta}}
\end{align*}
Then if $V$ is a Hodge-Tate representation then $\mbDHT(V)$ is a free $K_\Delta$-module of rank $\dim_F V$, and the canonical map
\begin{align*}
    \alpha_{\text{HT},\Delta}(V):\mBHT\otimes_K\mbDHT(V)\to \mBHT\otimes_{\mathbb{Q}_p}V
\end{align*}
is an isomorphism. The $K_\Delta$-module $\mbDHT(V)$ is a graded $K_\Delta$-module
\begin{align*}
    \mbDHT(V)=\bigoplus\limits_{i\in\mathbb{Z}}(\widehat{\C}_\Delta(i)\otimes_{\mathbb{Q}_p}V)^{\G_{K,\Delta}}
\end{align*}
and clearly $(\widehat{\C}_\Delta(i)\otimes_{\mathbb{Q}_p}V)^{\G_{K,\Delta}}$ are free $K_\Delta$-modules.

We also define for any $p$-adic representation $V$
\begin{align*}
    \mbDdR(V):=(\mBdr\otimes_{\mathbb{Q}_p}V)^{\G_{K,\Delta}}
\end{align*}
and the map
\begin{align*}
    \alpha_{\text{dR},\Delta}(V):\mBdr\otimes_K\mbDdR(V)\to \mBdr\otimes_{\mathbb{Q}_p}V
\end{align*}

If the representation $V$ is de Rham, then we have $\alpha_{\text{dR},\Delta}(V)$ is an isomorphism and $\mbDdR(V)$ is a free $K_\Delta$-module of rank $\dim_F(V)$.\\

If $V$ is any $p$-adic representation of $\G_{K,\Delta}$, then $\mbDdR(V)$ is a filtered $K_\Delta$-module with
\begin{align*}
    \Fil^i_\Delta\mbDdR(V):=(\Fil^i_\Delta\mBdr\otimes_{\mathbb{Q}_p}V)^{\G_{K,\Delta}}.
\end{align*}
Consider the short exact sequence
\begin{align*}
    0\to\Fil^{i+1}_\Delta\mBdr\to \Fil^i_\Delta\mBdr\to \widehat{\C}_\Delta(i)\to 0
\end{align*}

Tensoring with $V$ and taking the $\G_{K,\Delta}$-fixed points, we get 
\begin{align*}
    \gr^i\mbDdR(V)\hookrightarrow\mbDHT(V).
\end{align*}

\begin{proposition}
    If a $p$-adic representation $V$ of $\G_{K,\Delta}$ is de Rham, then $V$ is Hodge-Tate. 
\end{proposition}

\begin{proof}
    To show that $V$ is Hodge-Tate, it suffices to prove that $(\widehat{\C}_\Delta(i)\otimes_{\mathbb{Q}_p} V)^{\G_{K,\Delta}}$ is a free $K_\Delta$-module when $V$ is de Rham. Consider the short exact sequence
    \begin{align*}
        0\to \Fil^{i+1}_\Delta\mBdr\to\Fil^i_\Delta\mBdr\to \widehat{\C}_\Delta(i)\to 0 
    \end{align*}
    of $\mathbb{Q}_p$-vector spaces. Tensoring with $V$ and taking the $\G_{K,\Delta}$-fixed points, we obtain the exact sequence
    \begin{align}\label{Exact sequence III}
        0\to (\Fil^{i+1}_\Delta\mBdr\otimes_{\mathbb{Q}_p}  V)^{\G_{K,\Delta}}\to (\Fil^i_\Delta\mBdr\otimes_{\mathbb{Q}_p} V)^{\G_{K,\Delta}}\to (\widehat{\C}_\Delta(i)\otimes_{\mathbb{Q}_p} V)^{\G_{K,\Delta}}
    \end{align}
    The de Rham representation $V$ of $\G_{K,\Delta}$ can be viewed as a representation of $\G_{K_s}$ through the map $\eta_s:\G_{K_s}\to\G_{K,\Delta}$ for each $s\in\mathfrak{S}$. Because free modules are locally free of constant rank, $V$ is a de Rham representation of $\G_{K_s}$. Locally, the sequence obtained from (\ref{Exact sequence III}) is right-exact since $V$ is de Rham as a representation of $\G_{K_s}$; hence, the sequence (\ref{Exact sequence III}) is right-exact globally. Therefore, when $V$ is de Rham, we obtain the short exact sequence
    \begin{align*}
         0\to (\Fil^{i+1}_\Delta\mBdr\otimes_{\mathbb{Q}_p}  V)^{\G_{K,\Delta}}\to (\Fil^i_\Delta\mBdr\otimes_{\mathbb{Q}_p} V)^{\G_{K,\Delta}}\to (\widehat{\C}_\Delta(i)\otimes_{\mathbb{Q}_p} V)^{\G_{K,\Delta}}\to 0
    \end{align*}
    of $K_\Delta$-modules. Since $\Fil^i_\Delta\mbDdR(V)=(\Fil^i_\Delta\mBdr\otimes_{\mathbb{Q}_p}V)^{\G_{K,\Delta}}$ is finitely presented and $\Fil^{i+1}_\Delta\mbDdR(V)=(\Fil^{i+1}_\Delta\mBdr\otimes_{\mathbb{Q}_p}V)^{\G_{K,\Delta}}$ is finitely generated, $\mbDHT(V)=(\widehat{\C}_\Delta(i)\otimes_{\mathbb{Q}_p}V)^{\G_{K,\Delta}}$ is finitely presented. Therefore, Conjecture \ref{conj} holds in this case. Furthermore, because it has a constant rank at each localization, it is free. 
\end{proof}

\begin{remark}
    The de Rham and Hodge-Tate representations in the sense of \cite{BCM} are also de Rham and Hodge-Tate, respectively, but the converse may not hold.
\end{remark}

\section{Final Remarks}
The crystalline multivariate period rings can be constructed analogously in multivariable setups. However, we do not pursue this further here, as the construction is highly parallel to that of \cite{Fontaine}. In fact, the proofs involved are largely a matter of adapting the notation of the classical theory.

\bibliographystyle{amsalpha}
\bibliography{name}
\end{document}